\documentclass[11pt]{article}
\usepackage{amsmath,amssymb,amsthm}
\usepackage[margin=1in]{geometry}
\usepackage{graphicx}
\usepackage{hyperref}
\usepackage{enumitem}
\usepackage{authblk}
\usepackage{float}
\usepackage{subfigure}
\numberwithin{equation}{section}
\newtheorem{theorem}{Theorem}[section]
\newtheorem{lemma}[theorem]{Lemma}
\newtheorem{assumption}[theorem]{Assumption}

\newtheorem{remark}[theorem]{Remark}

\newcommand{\norm}[1]{\left\| #1 \right\|}
\newcommand{\e}{\operatorname{e}}

\title{Order Reduction of Exponential Runge--Kutta Methods:\\
	Fourth-Order Schemes for Non-Commuting Operators}
	
\author[1]{Thi Tam Dang}
\author[2]{Pablo Alexei Gazca-Orozco}
\author[2]{Trung Hau Hoang\thanks{{
			Email address: Thi Tam Dang:
			tam.dang@helsinki.fi, Pablo Alexei Gazca-Orozco: gazca@karlin.mff.cuni.cz, Trung Hau Hoang: trung-hau.hoang@matfyz.cuni.cz
}}}
\affil[1]{Department of Mathematics and Statistics, University of Helsinki, Finland} 
\affil[2]{Faculty of Mathematics and Physics, Charles University, Czech Republic}

\date{}

\begin{document}
	
\maketitle

\begin{abstract}
This paper extends the convergence analysis of explicit exponential Runge--Kutta methods for linear parabolic problems $u'(t) + Au(t) = Bu(t)$---where $A$ generates an analytic semigroup and $B$ is relatively bounded with respect to $A$---from the third-order case to fourth-order schemes. By establishing the global error recursion relation and extending the defect-based analytical framework, we identify the terms responsible for stiff order reduction when $A$ and $B$ do not commute. Numerical experiments are performed on a non-commuting advection--diffusion problem to validate the theoretical results. Numerical tests using the classical four-stage, fourth-order schemes of Krogstad and Strehmel \& Weiner exhibit an observed convergence order of approximately 2.75, which matches the theoretical prediction from the convergence analysis. 
\end{abstract}



\section{Introduction}

Exponential integrators have become a popular class of numerical methods for the time integration of stiff differential equations, particularly those arising from the spatial discretization of partial differential equations \cite{HO2010, MinchevWright2005, Hau1, Hau2, Hau3}. The fundamental motivation behind exponential methods is to integrate the linear, stiff part of a system exactly using matrix exponentials and related $\varphi$-functions, thereby avoiding the severe step-size restrictions imposed by standard explicit Runge--Kutta schemes. Over the last two decades, the theoretical foundations of explicit exponential Runge--Kutta methods for semilinear parabolic problems have been firmly established \cite{HO2005, LuanOstermann2013, LuanOstermann2014}. 

Despite their popularity, a pervasive challenge in the application of high-order exponential integrators is the phenomenon of \emph{stiff order reduction}. Classical convergence theory relies on the assumption that the non-stiff nonlinear operator (or perturbation) behaves smoothly. However, when applied to practical initial-boundary value problems, order reduction frequently manifests due to spatial boundary incompatibilities, non-smooth initial data, or a lack of commutativity between spatial operators \cite{EinkemmerOstermann2014, HochbruckOstermann2005_parabolic, OstermannRoche1992}. In this work, we focus on establishing sharp, rigorous convergence bounds for problems governed by non-commuting, unbounded spatial operators. In particular, we are interested in the nonlinear problem of the form
\begin{equation}\label{eq98} 
	\begin{aligned}
		& \partial_{t}u(t,x) - \Delta u(t,x) = f(u(t,x), \nabla u(t,x)), \qquad (t,x) \in (0,T] \times \Omega, \\
		& u(0,x) = u_0(x), \qquad x \in \Omega,
	\end{aligned}  
\end{equation}
subject to suitable boundary conditions. Equations of this form arise in numerous applications across science and engineering (see, e.g.,~\cite{Ablowitz1979,ALLEN19791085,bailyn1994survey,bergman2011fundamentals,Bronsard1993,grindrod1996theory}). Here, $u(t,x)$ denotes the unknown function, $f$ is a nonlinear function of $u$ and its spatial gradient $\nabla u$, and $\Omega \subset \mathbb{R}^N$ is a bounded domain, where $N$ represents the spatial dimension. 

In a recent study~\cite{orig}, a rigorous error analysis was developed for explicit exponential Runge--Kutta methods up to classical order three applied to the linear evolutionary model problem with non-commuting operators,
\begin{equation}\label{eq:model}
	u'(t) + Au(t) = Bu(t), \qquad u(0) = u_0,
\end{equation}
posed on a Banach space $X$ (which represents a simplified form of~\eqref{eq98}). Here, $A: D(A) \to X$ is an unbounded sectorial operator generating an analytic semigroup $\{e^{-tA}\}_{t \ge 0}$, while $B: D(B) \to X$ is an unbounded operator that is relatively bounded with respect to $A$ (such as an advection or lower-order spatial operator) and crucially does \emph{not} commute with $A$. The analysis in \cite{orig} demonstrated a sharp dichotomy between low- and high-order methods:
\begin{itemize}
	\item First-order (exponential Euler) and second-order eERK schemes retain their full classical order of convergence under mild regularity assumptions on the initial datum $u_0$.
	\item Third-order explicit exponential Runge--Kutta schemes suffer a systematic order reduction-dropping down to order $5/2$ in general, and reaching at most order $11/4$ for a specially tuned choice of internal stage coefficients.
\end{itemize}

A natural question left open in \cite{orig} is how this order reduction behaves at the fourth order. Four-stage, fourth-order explicit exponential Runge--Kutta methods satisfy a larger set of stiff order conditions. Crucially, whereas third-order conditions involve only a few non-commuting terms between $B$ and the coefficient functions $\varphi_k(-\tau A)$, several fourth-order stiff conditions introduce many more. Understanding how these non-commutative structures interact with analytic semigroup estimates and propagate through the global error recursion is essential to determine whether fourth-order schemes suffer further accuracy loss or hit a structural ceiling. 

In this paper, we extend the error-recursion machinery of \cite{orig} to four-stage, fourth-order explicit exponential Runge--Kutta schemes. We derive the full error recursion and classify all defect terms. We bound all relevant defect terms rigorously and formulate the final error bound in Theorem~\ref{theorem:main}. We validate the theoretical convergence results numerically using Krogstad's and Strehmel and Weiner’s classical fourth-order schemes \cite{HO2005}. The observed numerical order of $2.75$ confirms the predicted order reduction of fourth-order schemes.

This paper is organized as follows. Section \ref{sec:Analyticalframework} reviews the analytic semigroup framework in a Banach space~$X$ required for analyzing problem~\eqref{eq:model}. Section \ref{sec:Fourth-order} outlines the fourth-order explicit exponential Runge--Kutta schemes under consideration. In Section \ref{sec:recursion-full}, we present our rigorous convergence analysis along with the primary theoretical result. Section \ref{sec:numerics} provides numerical simulations to validate the order reduction bounds, and Section \ref{sec:conclusion} offers concluding remarks.

\section{Analytical framework}\label{sec:Analyticalframework}

To establish a rigorous environment for the convergence analysis, we adopt the functional analytical setting introduced in \cite{orig}, which we briefly recall here for completeness. Let $X$ denote a Banach space. Throughout the paper, $\norm{\cdot}$ denotes the norm in $X$ or the induced operator norm in $\mathcal{L}(X)$, as appropriate.

\begin{assumption}\label{ass:sectorial}
	The linear operator $A: D(A) \subset X \to X$ is sectorial. That is, $A$ is closed and densely defined, and there exist constants $\varphi \in (0,\pi/2)$, $M \in [1,\infty)$, and $a \in \mathbb{R}$ such that the resolvent set contains the open sector 
	\[
	S_{a,\varphi} = \big\{\lambda \in \mathbb{C} \setminus \{a\} : \varphi \le |\arg(\lambda - a)| \le \pi\big\},
	\]
	and the resolvent estimate
	\begin{equation}\label{eq:resolvent}
		\norm{(\lambda I - A)^{-1}} \le \frac{M}{|\lambda - a|}
	\end{equation}
	holds for all $\lambda \in S_{a,\varphi}$.
\end{assumption}

Under Assumption~\ref{ass:sectorial}, the operator $-A$ generates an analytic semigroup $\bigl\{\mathrm{e}^{-t A}\bigr\}_{t \ge 0}$. For any $\omega\in(-a,\infty)$, the fractional powers of the shifted operator $\widetilde{A}=A+\omega I$ are well defined (cf.~\cite{henry1981geometric}). By shifting $A$ if necessary and subsequently denoting the shifted operator again by $A$, we may assume without loss of generality that the fractional powers of $A$ are well defined. In our error analysis, we will frequently rely on the following standard parabolic smoothing estimate (see, e.g., \cite{HOCHBRUCK2005323}).

\begin{lemma}\label{lem1}
	Under Assumption~\ref{ass:sectorial}, there exists a constant $C > 0$ such that
	\begin{equation}\label{parabolicsmoothing1}
		\bigl\|\mathrm{e}^{-t A}\bigr\| + \bigl\|t^\gamma A^\gamma \mathrm{e}^{-t A}\bigr\| \le C
	\end{equation}
	holds for all $\gamma \ge 0$ uniformly in $t \in [0,T]$.
\end{lemma}

To account for the non-commuting operator $B$, we impose a relative boundedness condition defined via fractional power domains as follows.

\begin{assumption}\label{ass:B}
	The linear operator $B : D(B) \subset X \to X$ is closed, and there exists an exponent $\gamma \in (0, 1]$ such that $D(A^\gamma) \subset D(B)$. In particular, $B$ is relatively bounded with respect to $A^\gamma$, so that
	\begin{equation}\label{eq:B-rel-bound}
		\norm{BA^{-\gamma}} \le C \quad \text{and} \quad \norm{A^{-\gamma}B} \le C.
	\end{equation}
\end{assumption}

Under Assumptions~\ref{ass:sectorial} and~\ref{ass:B}, the perturbation theorem for sectorial operators guarantees that the perturbed operator $A - B$ is likewise sectorial \cite{pazy1983semigroups}. Therefore, the stability estimate \eqref{parabolicsmoothing1} extend directly to $A - B$. Throughout this paper, we assume that \(u\) possesses sufficient smoothness and that both \(u\) and its derivatives belong to the domain of \(B\).

Finally, following \cite{orig}, we employ the shorthand $M^- = M - \zeta$ and $M^+ = M + \tilde{\zeta}$ for arbitrarily small, fixed parameters $\zeta, \tilde{\zeta} > 0$. This notation convenienty absorbs any logarithmic factor losses that arise when estimating the defect recursions.
	
\section{Fourth-order exponential Runge--Kutta methods}\label{sec:Fourth-order}

When applied to the semi-discrete linear problem $u'(t) + Au(t) = Bu(t)$ with time step $\tau > 0$, an $s$-stage explicit exponential Runge--Kutta method computes internal stage approximations $U_{ni} \approx u(t_n + c_i\tau)$ and the updated solution $u_{n+1} \approx u(t_{n+1})$ via
\begin{subequations}\label{eq:scheme4}
	\begin{align}
		U_{ni} &= e^{-c_i \tau A}u_n + \tau\sum_{j=1}^{i-1} a_{ij}(-\tau A) B U_{nj}, \qquad 1 \le i \le s,\\
		u_{n+1} &= e^{-\tau A}u_n + \tau\sum_{i=1}^{s} b_i(-\tau A) B U_{ni},
	\end{align}
\end{subequations}
where $c_1 = 0$, $U_{n1} = u_n$, and the stage nodes satisfy $c_i \in (0,1]$ for $i \ge 2$. Here, we focus on four-stage ($s=4$), classical fourth-order schemes. The operator-valued coefficients $a_{ij}(-\tau A)$ and $b_i(-\tau A)$ are linear combinations of the matrix functions $\varphi_k(-\tau A)$ ($k \ge 0$), defined inductively by
\[
\varphi_0(z) = \operatorname{e}^z, \qquad \varphi_{k+1}(z) = \frac{\varphi_k(z) - \frac{1}{k!}}{z}, \quad (k \ge 0).
\]
Throughout the paper, we denote generic operator-valued coefficient combinations by $\phi(-\tau A)$. These coefficients satisfy the uniform bound
\begin{equation}\label{eq:phibound}
	\norm{\phi(-\tau A)} + \norm{\tau^{\eta}A^{\eta}\phi(-\tau A)} \le C, \qquad 0 \le \eta \le 1.
\end{equation}

To guarantee high-order convergence in the presence of stiffness, the coefficient functions must satisfy a set of stiff order conditions. Table~\ref{tab:order4} collects the algebraic conditions required for a four-stage method to achieve classical order four. In these expressions, $J$ and $K$ denote arbitrary bounded linear operators on $X$ 

\begin{table}[h!]

	\centering

	\begin{tabular}{cll}

		\hline

		No. & Order & Condition \\

		\hline

		1 & 1 & $\displaystyle\sum_{i=1}^{4} b_i(-\tau A) = \varphi_1(-\tau A)$ \\[4pt]

		2 & 2 & $\displaystyle\sum_{i=2}^{4} b_i(-\tau A)c_i = \varphi_2(-\tau A)$ \\[4pt]

		3 & 2 & $\displaystyle\sum_{j=1}^{i-1} a_{ij}(-\tau A) = c_i\varphi_1(-c_i\tau A)$ \\[4pt]

		4 & 3 & $\displaystyle\sum_{i=2}^{4} b_i(-\tau A)c_i^2 = 2\varphi_3(-\tau A)$ \\[4pt]

		5 & 3 & $\displaystyle\sum_{i=2}^{4} b_i(-\tau A)J\Big(\varphi_2(-c_i\tau A)c_i^2 - \sum_{j=2}^{i-1} a_{ij}(-\tau A)c_j\Big) = 0$ \\[4pt]

		6 & 4 & $\displaystyle\sum_{i=2}^{4} b_i(-\tau A)c_i^3 = 6\varphi_4(-\tau A)$ \\[4pt]

		7 & 4 & $\displaystyle\sum_{i=2}^{4} b_i(-\tau A)J\Big(\varphi_3(-c_i\tau A)c_i^3 - \tfrac{1}{2}\sum_{j=2}^{i-1} a_{ij}(-\tau A)c_j^2\Big) = 0$ \\[4pt]

		8 & 4 & $\displaystyle\sum_{i=2}^{4} b_i(-\tau A)J\sum_{j=2}^{i-1} a_{ij}(-\tau A)J\Big(\varphi_2(-c_j\tau A)c_j^2 - \sum_{k=2}^{j-1} a_{jk}(-\tau A)c_k\Big) = 0$ \\[4pt]

		9 & 4 & $\displaystyle\sum_{i=2}^{4} b_i(-\tau A)c_iK\Big(\varphi_2(-c_i\tau A)c_i^2 - \sum_{j=2}^{i-1} a_{ij}(-\tau A)c_j\Big) = 0$ \\

		\hline

	\end{tabular}

	\caption{Stiff order conditions for explicit four-stage exponential Runge--Kutta methods of order four.}

	\label{tab:order4}

\end{table} 
Here, we do not mandate that conditions~5, 7, 8, and~9 hold in their strong form. Instead, we enforce them in a \emph{very weak form}, evaluating \textbf{all} coefficient functions---including the outer weights $b_i(-\tau A)$, the internal stage weights $a_{ij}(-\tau A)$, and the functions $\varphi_k(-c_i\tau A)$---at $A = 0$.  We adopt the following foundational hypothesis generalizing the third-order formulation in \cite{orig}.

\begin{assumption}\label{ass:beta}
	There exists a parameter $\beta \in \mathbb{R}$ such that the operators $A^{\beta}BA^{-1}$ and $A^{-1}BA^{\beta}$ are bounded on $X$. 
\end{assumption}

\begin{remark}
Concrete examples of operators satisfying Assumptions \ref{ass:B}, \ref{ass:beta}, \ref{ass:sectorial}, will be provided after the statement of the main theorem, where the parameters introduced in these assumptions will be specified explicitly.
\end{remark}

\section{Error Recursion and Convergence}\label{sec:recursion-full}

We consider: $s=4$ (stages), $c_1=0$,
$U_{n1}=u_n$, coefficients $a_{ij}(-\tau A)$, $b_i(-\tau A)$ built from
$\varphi_k(-\tau A)$'s and the order conditions of Table~\ref{tab:order4}. 

%As in \cite{orig} we
%write $\Lambda(t)=t^{-1+\alpha+\zeta}$ or $t^{-1+\zeta}$ for the (interchangeable)
%weight coming from $u_0\in D(A^{2+\alpha})$, and we absorb arbitrarily small losses
%into $M_\pm$, $\zeta$.

\subsection{Taylor expansions: stages and update}

\paragraph{Stage expansion.} For $i=2,3,4$, applying the \emph{variation-of-constants} formula (see \cite{HO2010}) with $\theta=c_i$ and Taylor-expanding $Bu(t_n+\xi)$
in $\xi$ to order $2$ (remainder in $u'''$) gives the exact identity
\begin{equation}\label{eq:S}
	u(t_n+c_i\tau) = e^{-c_i\tau A}u(t_n) + \sum_{k=1}^{3}(c_i\tau)^k\varphi_k(-c_i\tau A)\,Bu^{(k-1)}(t_n) + \rho_i(t_n),
\end{equation}
\[
\rho_i(t_n) = \int_0^{c_i\tau} e^{-(c_i\tau-\xi)A}\int_0^\xi \frac{(\xi-\sigma)^2}{2}\,Bu'''(t_n+\sigma)\,d\sigma\,d\xi .
\]

\paragraph{Update expansion.} Expanding $Bu(t_n+\xi)$ to order $3$ (remainder in $u''''$) and using $\theta=1$:
\begin{equation}\label{eq:U}
	u(t_{n+1}) = e^{-\tau A}u(t_n) + \sum_{k=1}^{4}\tau^k\varphi_k(-\tau A)\,Bu^{(k-1)}(t_n) + \rho(t_n),
\end{equation}
\[
\rho(t_n) = \int_0^{\tau} e^{-(\tau-\xi)A}\int_0^\xi \frac{(\xi-\sigma)^3}{6}\,Bu''''(t_n+\sigma)\,d\sigma\,d\xi .
\]

\begin{remark}
	As in \cite{orig}, we distinguish two cases: $t_n = 0$ and $t_n \neq 0$.
	\begin{itemize}
		\item Case $t_n = 0$. 	Expanding $u(\xi)$ in a Taylor series about $0$ up to third order:
		\begin{equation}\label{errf32}
			u(\xi) = u(0) + \xi u'(0) + \frac{\xi^2}{2} u''(0) + \int_0^\xi \frac{(\xi - \sigma)^2}{2} u'''(\sigma) \, \mathrm{d}\sigma.
		\end{equation}
		\item Case $t_n \neq 0$. 	Expanding $u(t_n + \xi)$ in a Taylor series about $t_n$ up to fourth order:
		\begin{equation}\label{eq59}
			u(t_n + \xi) = u(t_n) + \xi u'(t_n) + \frac{\xi^2}{2} u''(t_n) + \frac{\xi^3}{6} u'''(t_n) + \int_0^\xi \frac{(\xi - \sigma)^3}{6} u^{(4)}(t_n + \sigma) \, \mathrm{d}\sigma.
		\end{equation}
		In particular, setting $\xi = c_j \tau$ for $j < i$ gives:
		\begin{equation}\label{eq:T}
			u(t_n + c_j\tau) = u(t_n) + c_j\tau u'(t_n) + \frac{(c_j\tau)^2}{2} u''(t_n) + \frac{(c_j\tau)^3}{6} u'''(t_n) + \int_0^{c_j\tau} \frac{(c_j\tau - \sigma)^3}{6} u^{(4)}(t_n + \sigma) \, \mathrm{d}\sigma.
		\end{equation}
	\end{itemize}
The expressions for the derivatives up to the fourth order are:
	\begin{equation}\label{eq16}
		u''(\sigma) = (A-B)^{1-\alpha} \mathrm{e}^{-\sigma(A-B)} (A-B)^{1+\alpha} u_0,
	\end{equation}
	\begin{equation}\label{eq16_3rd}
		u'''(t) = -(A-B)^{2-\alpha} \mathrm{e}^{-t(A-B)} (A-B)^{1+\alpha} u_0,
	\end{equation}
	\begin{equation}\label{eq16_4th}
		u^{(4)}(t) = (A-B)^{3-\alpha} \mathrm{e}^{-t(A-B)} (A-B)^{1+\alpha} u_0.
	\end{equation}
In the case $n = 1$ ($t_1 = 0$), the terms are evaluated using the variant \eqref{errf32} rather than \eqref{eq:T}.
\end{remark}

\subsection{Stage defects $\Delta_{ni}$}

Inserting the exact solution into stage $i$ of the scheme yields:
\begin{equation}\label{eq:6a}
	u(t_n+c_i\tau) = \mathrm{e}^{-c_i\tau A}u(t_n) + \tau\sum_{j=1}^{i-1}a_{ij}(-\tau A)\,Bu(t_n+c_j\tau) + \Delta_{ni}.
\end{equation}
  
To compute $\Delta_{ni}$, we substitute $u(t_n+c_j\tau)$ with $j = i$ from \eqref{eq:T} and subtract $u(t_n+c_i\tau)$ given by \eqref{eq:6a}. The $B u(t_n)$ terms cancel \emph{exactly} by order condition~3 ($\sum_{j<i}a_{ij}(-\tau A) = c_i\varphi_1(-c_i\tau A)$). Collecting the remaining terms yields:
\begin{equation}\label{eq:Delta-ni}
	\Delta_{ni} = \tau^2\,\mathcal J_{2,i}\,Bu'(t_n) \;+\;\tau^3\,\mathcal J_{3,i}\,Bu''(t_n) \;+\;\tau^4\,\mathcal J_{4,i}\,Bu'''(t_n) \;+\;\Xi_{ni},
\end{equation}
where the stage-level residual brackets are defined by:
\begin{align}
	\mathcal J_{2,i} &:= c_i^2 \varphi_2(-c_i\tau A) - \sum_{j=2}^{i-1}a_{ij}(-\tau A)\,c_j, \label{eq:J2}\\[4pt]
	\mathcal J_{3,i} &:= c_i^3 \varphi_3(-c_i\tau A) - \tfrac12\sum_{j=2}^{i-1}a_{ij}(-\tau A)\,c_j^2, \label{eq:J3}\\[4pt]
	\mathcal J_{4,i} &:= c_i^4 \varphi_4(-c_i\tau A) - \tfrac16\sum_{j=2}^{i-1}a_{ij}(-\tau A)\,c_j^3, \label{eq:J4}
\end{align}
and the localized integral remainder is now of order $\mathcal{O}(\tau^5)$:
\begin{equation}\label{eq:Xi}
	\Xi_{ni} = \rho_i(t_n) \;-\;\tau\sum_{j=1}^{i-1}a_{ij}(-\tau A)\,B\!\int_0^{c_j\tau}\frac{(c_j\tau-\sigma)^3}{3!}u^{(4)}(t_n+\sigma)\,\mathrm{d}\sigma.
\end{equation}

\subsection{Update defect $\delta_{n+1}$}

Inserting the \emph{exact} stage values into the update formula gives:
\begin{equation}\label{eq:6b}
	u(t_{n+1}) = \mathrm{e}^{-\tau A}u(t_n) + \tau\sum_{i=1}^{4}b_i(-\tau A)\,Bu(t_n+c_i\tau) + \delta_{n+1}.
\end{equation}
Although the order conditions can be satisfied in numerous ways, we specifically adopt the configuration where Conditions~1, 2, 4, and~6 hold in strong form, while Conditions~5, 7, 8, and~9 hold only in the \emph{very weak form} (at $A=0$). This choice is motivated by the fact that several well-known numerical schemes in the literature satisfy this precise structure, allowing us to validate the theoretical error bounds numerically in Section~\ref{sec:numerics}. Under this setting, substituting the variation-of-constants expansion into \eqref{eq:6b} and matching against the exact solution \eqref{eq:U} eliminates the terms involving $Bu(t_n)$, $Bu'(t_n)$, $Bu''(t_n)$, and $Bu'''(t_n)$, leaving the non-vanishing residuals as

%Substituting the variation-of-constants expansion for each $u(t_n+c_i\tau)$ into \eqref{eq:6b} and subtracting the exact solution $u(t_{n+1})$ \eqref{eq:U}, the terms $Bu(t_n)$, $Bu'(t_n)$, $Bu''(t_n)$, and $Bu'''(t_n)$ match the exact solution expansion whenever order conditions~1, 2, 4, and~6 hold in strong form. 
%
%However, since Conditions~5, 7, 8, and~9 hold only in the \emph{very weak form} (at $A=0$), the remaining non-zero residuals are given by
\begin{equation}\label{eq:delta}
	\begin{aligned}
		\delta_{n+1} &= \tau^3 \sum_{i=2}^4 b_i(-\tau A) J(\mathcal J_{2,i}) \, Bu'(t_n) 
		\;+\; \tau^4 \sum_{i=2}^4 b_i(-\tau A) J(\mathcal J_{3,i}) \, Bu''(t_n) \\
		& \qquad 
		\;+\; \tau^4 \sum_{i=2}^4 b_i(-\tau A) c_i K(\mathcal J_{2,i}) \, Bu'(t_n) 
		\;+\; \Xi_{n+1}.
	\end{aligned}
\end{equation}
The remaining Taylor-remainder term $\Xi_{n+1}$ is given by:
\begin{equation}\label{eq:Xi-update}
	\Xi_{n+1} = \rho(t_n) \;-\; \tau\sum_{i=1}^4 b_i(-\tau A) \rho_i(t_n).
\end{equation}

\subsection{Solving the recursion}

Let $e_n = u_n - u(t_n)$ denote the global error at step $n$, and let $E_{ni} = U_{ni} - u(t_n + c_i\tau)$ represent the internal stage error at stage $i$. To analyze the global error propagation, the stage errors $E_{ni}$ are substituted recursively into the global update formula for $e_{n+1}$. Recall the stage error recursion:
\begin{equation}\label{eq:E_ni_rec}
	E_{ni} = \mathrm{e}^{-c_i\tau A}e_n + \tau\sum_{j<i}a_{ij}(-\tau A)BE_{nj} - \Delta_{ni},
\end{equation}
and the global update error relation:
\begin{equation}\label{eq:e_n1_rec}
	e_{n+1} = \mathrm{e}^{-\tau A}e_n + \tau\sum_{i=1}^4 b_i(-\tau A)BE_{ni} - \delta_{n+1}.
\end{equation}

For an explicit four-stage method ($i=1,2,3,4$), writing out the internal stages step-by-step for $i=1,2,3,4$ gives
\begin{align*}
	E_{n1} &= e_n - \Delta_{n1}, \\[4pt]
	E_{n2} &= \mathrm{e}^{-c_2\tau A}e_n + \tau a_{21} B e_n - \tau a_{21} B \Delta_{n1} - \Delta_{n2}, \\[4pt]
	E_{n3} &= \mathrm{e}^{-c_3\tau A}e_n + \tau a_{31} B e_n + \tau a_{32} B \mathrm{e}^{-c_2\tau A}e_n + \tau^2 a_{32} B a_{21} B e_n \\
	&\quad - \tau a_{31} B \Delta_{n1} - \tau a_{32} B \Delta_{n2} - \tau^2 a_{32} B a_{21} B \Delta_{n1} - \Delta_{n3}, \\[4pt]
	E_{n4} &= \mathrm{e}^{-c_4\tau A}e_n + \tau a_{41} B e_n + \tau a_{42} B \mathrm{e}^{-c_2\tau A}e_n + \tau a_{43} B \mathrm{e}^{-c_3\tau A}e_n \\
	&\quad + \tau^2 a_{42} B a_{21} B e_n + \tau^2 a_{43} B a_{31} B e_n + \tau^2 a_{43} B a_{32} B \mathrm{e}^{-c_2\tau A}e_n \\
	&\quad + \tau^3 a_{43} B a_{32} B a_{21} B e_n \\
	&\quad - \tau \Big( a_{41} B \Delta_{n1} + a_{42} B \Delta_{n2} + a_{43} B \Delta_{n3} \Big) \\
	&\quad - \tau^2 \Big( a_{42} B a_{21} B \Delta_{n1} + a_{43} B a_{31} B \Delta_{n1} + a_{43} B a_{32} B \Delta_{n2} \Big) \\
	&\quad - \tau^3 a_{43} B a_{32} B a_{21} B \Delta_{n1} - \Delta_{n4}.
\end{align*}

Substituting these expressions into \eqref{eq:e_n1_rec} eliminates all stage variables $E_{ni}$, producing the explicit single-step error recursion:
\begin{equation}\label{eq:e_n-full}
	\begin{aligned}
			e_{n+1} &= \mathrm{e}^{-\tau A}e_n   + \tau \sum_{i=1}^4 b_i(-\tau A) B \mathrm{e}^{-c_i\tau A} e_n  \\
			&\quad + \tau^2 \sum_{i=2}^4 \sum_{j=1}^{i-1} b_i(-\tau A) B a_{ij}(-\tau A) B \mathrm{e}^{-c_j\tau A} e_n  \\
			&\quad + \tau^3 \sum_{i=3}^4 \sum_{j=2}^{i-1} \sum_{k=1}^{j-1} b_i(-\tau A) B a_{ij}(-\tau A) B a_{jk}(-\tau A) B \mathrm{e}^{-c_k\tau A} e_n  \\
			&\quad + \tau^4 b_4(-\tau A) B a_{43}(-\tau A) B a_{32}(-\tau A) B a_{21}(-\tau A) B e_n  \\
			&\quad - \tau \sum_{i=1}^4 b_i(-\tau A) B \Delta_{ni}  \\
			&\quad - \tau^2 \sum_{i=2}^4 \sum_{j=1}^{i-1} b_i(-\tau A) B a_{ij}(-\tau A) B \Delta_{nj}  \\
			&\quad - \tau^3 \sum_{i=3}^4 \sum_{j=2}^{i-1} \sum_{k=1}^{j-1} b_i(-\tau A) B a_{ij}(-\tau A) B a_{jk}(-\tau A) B \Delta_{nk}  \\
			&\quad - \delta_{n+1}.
	\end{aligned}
\end{equation}

Solving the error recursion \eqref{eq:e_n-full} and grouping terms by their structure yields the global error representation: 
\begin{equation} 
	e_n = \mathrm{I} + \mathrm{II} + \mathrm{III} + \mathrm{IV} + \mathrm{V} + \mathrm{VI} +  \mathrm{VII},
\end{equation}
where each group is explicitly given by:

\begin{itemize}[leftmargin=1.4em]
	\item $
		\mathrm{I} = \tau \sum_{j=0}^{n-1} \mathrm{e}^{-(n-j-1)\tau A}  \left( \sum_{i=1}^4 b_i(-\tau A) B \mathrm{e}^{-c_i\tau A} \right) e_j,
	$
	
	\item $
		\mathrm{II} = \tau^2 \sum_{j=0}^{n-1} \mathrm{e}^{-(n-j-1)\tau A} \left( \sum_{i=2}^4 \sum_{k=1}^{i-1} b_i(-\tau A) B a_{ik}(-\tau A) B \mathrm{e}^{-c_k\tau A} \right) e_j,
	$
	
	\item $
		\mathrm{III} = \tau^3 \sum_{j=0}^{n-1} \mathrm{e}^{-(n-j-1)\tau A} \left( \sum_{i=3}^4 \sum_{k=2}^{i-1} \sum_{m=1}^{k-1} b_i(-\tau A) B a_{ik}(-\tau A) B a_{km}(-\tau A) B \mathrm{e}^{-c_m\tau A} \right) e_j,
	$

	\item $
		\mathrm{IV} = \tau^4 \sum_{j=0}^{n-1} \mathrm{e}^{-(n-j-1)\tau A} \Big( b_4(-\tau A) B a_{43}(-\tau A) B a_{32}(-\tau A) B a_{21}(-\tau A) B \Big) e_j.
$

\item The quantities $T_n^{[1]}$, $T_n^{[2]}$, and $T_n^{[3]}$ are defined as
\begin{align}
	T_n^{[1]} &:= \sum_{i=2}^4 b_i(-\tau A)B \mathcal J_{2,i} \, B u'(t_{n-1}),
%	\qquad &&(\text{Condition 5}), \label{eq:T1}
	\\[4pt] 
	T_n^{[2]} &:= \sum_{i=2}^4 b_i(-\tau A) B \mathcal J_{3,i} \, B u''(t_{n-1}),
%	\qquad &&(\text{Condition 7}), \label{eq:T2}
	\\[4pt]
	T_n^{[3]} &:=  \sum_{i=2}^4 c_i b_i(-\tau A) B \mathcal J_{2,i} \, B u'(t_{n-1}).
%	\qquad &&(\text{Condition 9}). \label{eq:T3}
\end{align}
Error group V takes the form
\begin{equation}\label{eq:VI-def}
	\mathrm{V} = -\tau^3 \sum_{j=0}^{n-1} \mathrm{e}^{-j\tau A} T_{n-j}^{[1]}
	\;-\; \tau^4 \sum_{j=0}^{n-1} \mathrm{e}^{-j\tau A}
	\left( T_{n-j}^{[2]} + T_{n-j}^{[3]} \right).
\end{equation}

Here, $T_n^{[1]}$, $T_n^{[2]}$, and $T_n^{[3]}$ vanish whenever the
corresponding order condition (5, 7, and 9, respectively) holds in the
\emph{strong} form of Table~\ref{tab:order4}. If a condition holds only in
the weak or very weak sense, the corresponding $T_n^{[l]}$
is generally nonzero and must be estimated, as shown below.

\begin{remark}\label{rem:krogstad-VI}
	For Krogstad's scheme, as well as Strehmel and Weiner's scheme specifically, Conditions~5 and~9 hold in general, while Condition~7 holds only in a very weak form (with all arguments evaluated at $A=0$). Consequently,
	\[
	T_n^{[1]} = 0, \qquad T_n^{[3]} = 0,
	\]
	and \eqref{eq:VI-def} collapses to
	\begin{equation}\label{eq:VI-krogstad}
\mathrm{V}\big|_{\text{Krogstad / Strehmel--Weiner}} = -\tau^4 \sum_{j=0}^{n-1} \mathrm{e}^{-j\tau A}\, T_{n-j}^{[2]}.
	\end{equation}
	The bound derived below for $T_n^{[2]}$ is therefore the
	\emph{only} contribution from Group V that exists for Krogstad's scheme as well as Strehmel and Weiner's scheme;
	the estimates for $T_n^{[1]}$ and $T_n^{[3]}$ are retained in the general
	analysis for the sake of other four-stage schemes that satisfy Conditions~5
	and/or~9 only weakly.
\end{remark}

\item $
	\mathrm{VI} = -\sum_{j=0}^{n-1} \mathrm{e}^{-j\tau A} R_{n-j}.
$ 
Here, we have
\begin{equation}\label{eq:R1}
	R_n^{[k]} := \tau\, b_{k+1}(-\tau A) B \int_0^{c_{k+1}\tau}
	\mathrm{e}^{-(c_{k+1}\tau-\xi)A} \int_0^\xi \frac{(\xi-\sigma)^2}{2!}\,
	B u'''(t_{n-1}+\sigma) \, \mathrm d\sigma \, \mathrm d\xi,
	\,  k=1,2,3,
\end{equation}
%so that $R_n^{[1]}, R_n^{[2]}, R_n^{[3]}$ correspond respectively to
%$k+1=2,3,4$. The update-step remainder contributes the fourth, quartic-kernel
%term
\begin{equation}\label{eq:R5}
	R_n^{[4]} := \int_0^\tau \mathrm{e}^{-(\tau-\xi)A} \int_0^\xi
	\frac{(\xi-\sigma)^3}{3!}\, B u''''(t_{n-1}+\sigma) \, \mathrm d\sigma \, \mathrm d\xi,
\end{equation}
%alongside the three remaining $b_i(-\tau A)$-weighted step-type remainders
\begin{align}
R_n^{[5]} &:= -\tau\,b_2(-\tau A)\int_0^{c_2\tau} \frac{(c_2\tau-\sigma)^3}{3!}\,
B u''''(t_{n-1}+\sigma)\,\mathrm d\sigma, \qquad n\ge 2, \label{eq:R5} \\[4pt]
	R_n^{[6]} &:= -\tau\,b_3(-\tau A)\int_0^{c_3\tau} \frac{(c_3\tau-\sigma)^3}{3!}\,
	B u''''(t_{n-1}+\sigma)\,\mathrm d\sigma, \qquad n\ge 2, \label{eq:R6}\\[4pt]
	R_n^{[7]} &:= -\tau\,b_4(-\tau A)\int_0^{c_4\tau} \frac{(c_4\tau-\sigma)^3}{3!} \,
	B u''''(t_{n-1}+\sigma)\,\mathrm d\sigma, \qquad n\ge 2, \label{eq:R7} \\ 
	R_1^{[5]} &:= -\tau\,b_2(-\tau A)\int_0^{c_2\tau} \frac{(c_2\tau-\sigma)^2}{2!}\,
	B u'''(\sigma)\,\mathrm d\sigma, \label{eq:R5-boundary}\\
	R_1^{[6]} &:= -\tau\,b_3(-\tau A)\int_0^{c_3\tau} \frac{(c_3\tau-\sigma)^2}{2!}\,
	B u'''(\sigma)\,\mathrm d\sigma, \label{eq:R6-boundary}\\
	R_1^{[7]} &:= -\tau\,b_4(-\tau A)\int_0^{c_4\tau} \frac{(c_4\tau-\sigma)^2}{2!}\,
	B u'''(\sigma)\,\mathrm d\sigma. \label{eq:R7-boundary}
\end{align}

\item Define $W_n$ by 
\begin{equation}\label{eq:Wn-def}
	W_n := \sum_{i=3}^4 \sum_{j=2}^{i-1} b_i(-\tau A) \, B \, a_{ij}(-\tau A)
	\, B \, \mathcal J_{2,j} \, B u'(t_{n-1}).
\end{equation} 
Group VII takes the compact form:
\begin{equation}\label{eq:VIII-def}
	\mathrm{VII} = -\tau^4 \sum_{l=0}^{n-1} \mathrm{e}^{-l\tau A} W_{n-l}.
\end{equation}
\end{itemize}
	 
\subsection{Bounds for Groups I--VII and Main Results}
	
Each term is estimated by inserting $A^{\gamma}A^{-\gamma}$ at every occurrence of $B$, then grouping $A^{-\gamma}B$ to apply Assumption~\ref{ass:B}.

\paragraph{Group I--IV.}
For $j\le n-2$, writing
$$\e^{-(n-j-1)\tau A} \left( \sum_{i=1}^4 b_i(-\tau A) B \mathrm{e}^{-c_i\tau A} \right)
=(\e^{-(n-j-1)\tau A}A^{\gamma})\, \left(  \sum_{i=1}^4 b_i(-\tau A)\,(A^{-\gamma}B)\,\e^{-c_i\tau A}\right)$$
and applying \eqref{lem1}, \eqref{eq:phibound}, Assumption~\ref{ass:B} gives
$$  \left\|  (\e^{-(n-j-1)\tau A}A^{\gamma}) \left(  \sum_{i=1}^4 b_i(-\tau A)\,(A^{-\gamma}B)\,\e^{-c_i\tau A}\right) \right\| \le Ct_{n-j-1}^{-\gamma}.$$ For $j=n-1$, using
$\| A^{\gamma}b_i(-\tau A) \| \le C\tau^{-\gamma}, \, \forall i = \overline{1,4}$ instead gives $$\left\|    \sum_{i=1}^4 A^{\gamma} b_i(-\tau A)\,(A^{-\gamma}B)\,\e^{-c_i\tau A}  \right\| \le C\tau^{-\gamma}.$$ Hence, we get

\begin{equation}
	\|\mathrm I\|\le C\tau\sum_{j=0}^{n-2}t_{n-j-1}^{-\gamma}\|e_j\|
	+ C\tau^{1-\gamma}\|e_{n-1}\|. \label{eq:groupI}
\end{equation}

Concerning the estimates for $\mathrm{II}$, we observe that $\mathrm{II}$ has essentially the same structure as $\mathrm{I}$. The only difference is the additional factor $\tau a_{ik}(-\tau A)B$ appearing in $\mathrm{II}$, which results in a loss of order $\gamma$. Indeed, using
\begin{equation}
	\tau \left\| a_{ik}\left(-\tau A\right)B \right\|
	\leq \tau^{1-\gamma},
\end{equation}
we obtain
\begin{equation}
	\begin{aligned}
		\left\|\mathrm{II}\right\|
		&\leq C\tau^{1-\gamma}
		\left(
		\tau \sum_{j=0}^{n-2}
		t_{n-j-1}^{-\gamma}\left\|e_j\right\|
		+\tau^{1-\gamma}\left\|e_{n-1}\right\|
		\right) \\
		&\leq C\tau
		\sum_{j=0}^{n-2}
		t_{n-j-1}^{-\gamma}\left\|e_j\right\|
		+ C\tau^{1-\gamma}\left\|e_{n-1}\right\|.
	\end{aligned}
\end{equation}
In the last step, we have used $\tau<1$, which entails
$\tau^{1-\gamma}\leq 1$ for $\gamma\in[0,1)$; this restriction is harmless in view of the asymptotic regime $\tau\to0$. The estimates for $\mathrm{III}$ and $\mathrm{IV}$ follow in the same manner, and we therefore omit the analogous calculations.

\paragraph{Group V.}

Before deriving the estimates, we first recall two key lemmas from \cite{orig}.
\begin{lemma}\label{lem:u2u3}
	Under Assumptions~\ref{ass:sectorial}--\ref{ass:B} and $u_0\in D(A^{1+\alpha})$,
	$0<\alpha\le\tfrac12$,
	\[
	\|u''(t)\|\le Ct^{-1+\alpha}, \qquad \|u'''(t)\|\le Ct^{-2+\alpha},  \qquad \|u''''(t)\|\le Ct^{-3+\alpha},\qquad t\in(0,T].
	\]
\end{lemma} 

\begin{lemma}\label{lem:bd4}
	Under Assumptions~\ref{ass:sectorial}--\ref{ass:beta} and $u_0\in D(A^{1+\alpha})$,
	$0<\alpha\le\tfrac12$,
	\[
	\|B\phi(-\tau A)Bu'(t)\|\le C\tau^{\beta-\gamma}t^{-1+\alpha},\qquad
	\|B\phi(-\tau A)Bu''(t)\|\le C\tau^{\beta-\gamma}t^{-2+\alpha},
	\]
	for any linear combination $\phi(-\tau A)$ of the $\varphi_k(-\tau A)$.
\end{lemma}
\begin{proof}
The first bound follows from \cite{orig}. The second bound follows analogously by using
$u''(t)=-(A-B)^{-1}u'''(t)$ together with
$\|u'''(t)\|\le Ct^{-2+\alpha}$.
\end{proof}
We now proceed with the estimates, starting with the bounds for $T_n^{[1]}$ and $T_n^{[2]}$. We treat the three regimes $j=0$, $j=n-1$, and $1\le j\le n-2$ separately. We consider the case
$j=0$. Since $t_{n-1}>0$, Lemma~\ref{lem:bd4} applies directly:
\begin{equation}\label{eq:t1t2}
	\begin{aligned}
\tau^3 \| T_n^{[1]}\| + \tau^4 \| T_n^{[2]}\| & \le \tau^3 \sum_{i=2}^4 \|b_i(-\tau A)\|\,\|B \mathcal J_{2,i} Bu'(t_{n-1})\| \\
& \qquad + \tau^4 \sum_{i=2}^4 \|b_i(-\tau A)\|\,\|B \mathcal J_{3,i} Bu''(t_{n-1})\| \\
& 
\le Ct_{n-1}^{-1+\alpha} \,\tau^{3+\beta-\gamma} + Ct_{n-1}^{-2+\alpha} \,\tau^{4+\beta-\gamma} .
	\end{aligned}
\end{equation}
Similarly, the term corresponding to $j=n-1$ can be estimated by 
{\allowdisplaybreaks
	\begin{align*}
		\tau^3 \left \| \mathrm{e}^{-(n-1) \tau A}  \mathcal{T}_{1}^{[1]}   \right \|  + \tau^4 \left \| \mathrm{e}^{-(n-1) \tau A}  \mathcal{T}_{1}^{[2]}   \right \|
		& \leq  C  \tau^3 t_{n-1}^{-1+\alpha} \sum_{i=2}^4 \left \| A^{-1} B A^{\beta} \right \| \left \|  A^{-\beta +\frac{1}{2}}  \mathcal J_{2,i}  \right \|  \\
		& \qquad + C \tau^4 t_{n-1}^{-1+\alpha} \sum_{i=2}^4 \left \| A^{-1} B A^{\beta} \right \| \left \|  A^{-\beta +\frac{1}{2}}  \mathcal J_{3,i}  \right \|  \\
		&  \leq C t_{n-1}^{-1+\alpha} \tau^{3 + \beta-\gamma} + C t_{n-1}^{-1+\alpha} \tau^{4 + \beta-\gamma} .		
	\end{align*}
}We next examine the interior sum. We use the following standard result from \cite{HO2005}: there exists a bounded operator $\widehat{\phi}(-\tau A)$ such that
\begin{equation}\label{eq324}
\phi(-\tau A)-\phi(0)=(\tau A)^{\widehat{\Gamma}} \widetilde{\phi}(-\tau A),
\qquad 0 \leq \widehat{\Gamma} \leq 1 .
\end{equation}
Using Conditions 5 and 7 at $A=0$, together with \eqref{eq324} for $\phi = b_2(-\tau A), b_3(-\tau A), b_4(-\tau A),$ $\varphi_2(-\tau A), \varphi_3(-\tau A),$ and $\varphi_4(-\tau A)$, we can write the sum of the remaining terms, with $j\neq 0$ and $j\neq n-1$, as follows
{\allowdisplaybreaks
\begin{equation}\label{eq21}
	\begin{aligned}
		&  \tau^3 \sum_{j=1}^{n-2} \mathrm{e}^{-j \tau A}   \mathcal{T}_{n-j}^{[1]}  + \tau^3 \sum_{j=1}^{n-2} \mathrm{e}^{-j \tau A} \mathcal{T}_{n-j}^{[2]}  
		\\
		& =  \tau^{3+\Gamma} \sum_{j=1}^{n-2}  \left( \mathrm{e}^{-j \tau A}   A^{\Gamma } \right)  \sum_{i=2}^4 \widetilde{b_i}(-\tau A)   \left ( B   \mathcal J_{2,i} \left(- \tau A\right) B u^{\prime }\left(t_{n-j-1}\right) \right )  \\
		& \quad	+	 \tau^{4+\Gamma} \sum_{j=1}^{n-2} \left ( \mathrm{e}^{-j \tau A}    A^{\Gamma } \right )  \sum_{i=2}^4  \widetilde{b_i}(-\tau A)  \left ( B \mathcal J_{3,i}  B u^{\prime \prime }\left(t_{n-j-1}\right) \right )  \\
		& \quad +  \tau^{3 + 2\beta}  \sum_{j=1}^{n-2} \left ( \mathrm{e}^{-j \tau A} A \right )   A^{-1} B  A^{\beta} \sum_{i=2}^4  \left (  \widetilde{\varphi_i}(-\tau A) \right )  A^{\beta} B A^{-1}  \left ( A (A-B)^{-1} \right )  u^{\prime \prime }\left(t_{n-j-1}\right)  \\
		& \quad	+  \tau^{4 + 2\beta} \sum_{j=1}^{n-2} \left ( \mathrm{e}^{-j \tau A}   A  \right )  A^{-1}  B A^{\beta} \sum_{i=2}^4 \left (  \widetilde{\varphi_i}(-\tau A) \right )  A^{\beta} B A^{-1}  \left ( A (A-B)^{-1} \right )  u^{\prime  \prime \prime}\left(t_{n-j-1}\right) . 	\end{aligned} 	
\end{equation}
}
Using \eqref{parabolicsmoothing1}, and Lemmas \ref{lem:u2u3}, \ref{lem:bd4} to estimate each term of the parentheses in \eqref{eq21}, we obtain
{\allowdisplaybreaks
	\begin{equation}\label{eq75}
		\begin{aligned}
			\left \|  \tau^3 \sum_{j=1}^{n-2} \mathrm{e}^{-j \tau A}   \mathcal{T}_{n-j}^{[1]}  + \tau^3 \sum_{j=1}^{n-2} \mathrm{e}^{-j \tau A} \mathcal{T}_{n-j}^{[2]}  \right \|
			& \leq C \tau^{ 3 + \Gamma+ \beta-\gamma}  \sum_{j=1}^{n-2}	t_{j}^{- \Gamma } t_{n-j-1}^{-1+\alpha} + C \tau^{4+2\beta}  \sum_{j=1}^{n-2}	t_{j}^{-1} t_{n-j-1}^{-2+\alpha} \\
			& \leq C t_n^{-\Gamma + \alpha} \tau^{3 + \Gamma+ \beta-\gamma } + C t_n^{-1  +\zeta } \tau^{2 +2\beta + \alpha-\zeta} \\
			& \leq C t_n^{-1 + \zeta + \alpha} \tau^{3 - \zeta+ \beta-\gamma } + C t_n^{-1  +\zeta } \tau^{2 +2\beta+ \alpha-\zeta} \text{ for } \Gamma = 1 -\zeta.					
		\end{aligned}	
	\end{equation}
}We now bound $T_n^{[3]}$ using Condition~9. However, the term $T_n^{[3]}$ differs from $T_n^{[1]}$ only by the bounded scalar $c_i$. Hence, it satisfies the same bound as $T_n^{[1]}$, with a prefactor of $\tau^4$. Therefore, the convergence rate of $T_n^{[3]}$ is even higher than that of $T_n^{[1]}$. Consequently, $T_n^{[3]}$ does not affect the overall error estimate, and we may safely omit its detailed analysis.
In summary, we get the following combined bound for
\begin{equation}\label{eq:VI-final}
	\|\mathrm{V}\|
	\le C\,t_n^{-1+\alpha+\zeta}\,\tau^{\,3-\zeta+\beta-\gamma}
	+ C\,t_n^{-1+\zeta}\,\tau^{\,2+2\beta+\alpha-\zeta}.
\end{equation}
\begin{remark}
	For Group V, consider the Krogstad and Strehmel--Weiner schemes.
	Recall that, in this case, $T_n^{[1]} = T_n^{[3]} = 0$.
	Combining these bounds and collecting all cases, we obtain
	\begin{equation}\label{eq:VI-krogstad-bound}
		\mathrm{V}\big|_{\text{Krogstad / Strehmel--Weiner}}
		\le
		C\,t_n^{-1+\zeta}\,\tau^{\,2+2\beta+\alpha-\zeta}.
	\end{equation}
	For brevity, we refer the reader to the worked example in \cite{orig} for the values
	$\alpha=\tfrac14-\zeta$, $\beta=\tfrac14-\zeta$, and $\gamma=\tfrac12$. The exponent in
	\eqref{eq:VI-krogstad-bound} becomes
	\[
	2+2\beta+\alpha
	= 2 + 2\Big(\tfrac14-\zeta\Big) + \Big(\tfrac14-\zeta\Big)
	= \tfrac{11}{4} - 3\zeta.
	\]
	Thus, Group V contributes an order arbitrarily close to $\tfrac{11}{4}$
	for the Krogstad and Strehmel--Weiner schemes under these parameter values.
	This matches, for this particular group, the order $\tfrac{11}{4}$
	stated in Theorem~\ref{theorem:main} and observed numerically in
	\S\ref{sec:numerics}.
\end{remark}
\paragraph{Group VI}
We first recall the following lemma from \cite{orig}.
\begin{lemma}\label{lem:coeffB}
	Under Assumptions~\ref{ass:sectorial} and \ref{ass:B}, for any bounded
	linear combination $\phi(-\tau A)$ of the $\varphi_k(-\tau A)$ (in particular,
	any $a_{ij}(-\tau A)$ or $b_i(-\tau A)$, $1\le i,j\le 4$),
	\begin{align}
		\|\phi(-\tau A)B\| &\le C\tau^{-\gamma}, \label{eq:coeffB-a}\\
		\|e^{-tA}\phi(-\tau A)B\| &\le Ct^{-\gamma}, \qquad t\in(0,T]. \label{eq:coeffB-b}
	\end{align}
\end{lemma}

We estimate $\mathrm{VI}=-\sum_{j=0}^{n-1}e^{-j\tau A}R_{n-j}$ by considering two representative terms: the stage-type term $R_n^{[1]}$ (with a cubic kernel and $u'''$) and the step-type term $R_n^{[5]}$ (with a quartic kernel and $u''''$). The remaining five terms in $R_n=\sum_{l=1}^7R_n^{[l]}$ can be bounded similarly; see \cite{orig}. We split the sum over $j$ into three regimes: $j=0$, $j=n-1$, and $1\le j\le n-2$.

For the case $j=0$, using \eqref{eq:phibound} together with Lemmas~\ref{lem:u2u3} and~\ref{lem:coeffB}, we obtain 
\[
\begin{aligned}
	\|R_n^{[1]}\| &\le \tau\,\|b_{2}(-\tau A)B\|
	\int_0^{c\tau}\!\!\int_0^\xi \|\e^{-(c\tau-\xi)A}B\|\,\frac{(\xi-\sigma)^2}{2}\,
	\|u'''(t_{n-1}+\sigma)\|\,\mathrm d\sigma\,\mathrm d\xi \\
	&\le C\,t_{n-1}^{-2+\alpha}\,\tau\int_0^{c\tau}\!\!\int_0^\xi (c\tau-\xi)^{-\gamma}(\xi-\sigma)^2\,\mathrm d\sigma\,\mathrm d\xi
	\le C\,t_{n-1}^{-2+\alpha}\,\tau^{5-\gamma},
\end{aligned}
\]
\[
\begin{aligned}
	\|R_n^{[5]}\| &\le \int_0^\tau\!\!\int_0^\xi \|\e^{-(\tau-\xi)A}B\|\,\frac{(\xi-\sigma)^3}{6}\,
	\|u''''(t_{n-1}+\sigma)\|\,\mathrm d\sigma\,\mathrm d\xi \\
	&\le C\,t_{n-1}^{-3+\alpha}\int_0^\tau\!\!\int_0^\xi (\tau-\xi)^{-\gamma}(\xi-\sigma)^3\,\mathrm d\sigma\,\mathrm d\xi
	\le C\,t_{n-1}^{-3+\alpha}\,\tau^{5-\gamma}.
\end{aligned}
\]
For the case $j=n-1$, the term is $\e^{-(n-1)\tau A}R_1$, where $R_1$ is defined at $t_0=0$ using the lower-order Taylor truncation, with $u''$ and $u'''$ in place of $u'''$ and $u''''$, respectively:
\[
R_1^{[1]} := \tau\,b_2(-\tau A)B\int_0^{c\tau}e^{-(c\tau-\xi)A}\int_0^\xi (\xi-\sigma)\,Bu''(\sigma)\,\mathrm d\sigma\,\mathrm d\xi,
\]
\[
R_1^{[5]} := \int_0^\tau e^{-(\tau-\xi)A}\int_0^\xi \frac{(\xi-\sigma)^2}{2}\,Bu'''(\sigma)\,\mathrm d\sigma\,\mathrm d\xi.
\]
Using Lemmas~\ref{lem:u2u3} and~\ref{lem:coeffB}, we get 
\[
\begin{aligned}
	\|\e^{-(n-1)\tau A}R_1^{[1]}\| &\le \tau\,\|\e^{-(n-1)\tau A}b_2(-\tau A)B\|
	\int_0^{c\tau}\!\!\int_0^\xi \|\e^{-(c\tau-\xi)A}B\|\,(\xi-\sigma)\,\|u''(\sigma)\|\,\mathrm d\sigma\,\mathrm d\xi \\
	&\le C\,t_{n-1}^{-\gamma}\,\tau\int_0^{c\tau}\!\!\int_0^\xi (c\tau-\xi)^{-\gamma}(\xi-\sigma)\,\sigma^{-1+\alpha}\,\mathrm d\sigma\,\mathrm d\xi
	\le C\,t_{n-1}^{-\gamma}\,\tau^{3+\alpha-\gamma},
\end{aligned}
\]
\[
\begin{aligned}
	\|\e^{-(n-1)\tau A}R_1^{[5]}\| &\le \|\e^{-(n-1)\tau A}A^\gamma\|
	\int_0^\tau\!\!\int_0^\xi \|\e^{-(\tau-\xi)A}\|\,\|A^{-\gamma}B\|\,\frac{(\xi-\sigma)^2}{2}\,\|u'''(\sigma)\|\,\mathrm d\sigma\,\mathrm d\xi \\
	&\le C\,t_{n-1}^{-\gamma}\int_0^\tau\!\!\int_0^\xi (\xi-\sigma)^2\,\sigma^{-2+\alpha}\,\mathrm d\sigma\,\mathrm d\xi
	\le C\,t_{n-1}^{-\gamma}\,\tau^{4+\alpha}.
\end{aligned}
\]
For the case $1\le j\le n-2$, the interior sum, consisting of the terms with $j\neq0$ and $j\neq n-1$, can be bounded using Assumption~\ref{ass:beta}, \eqref{eq:phibound}, and Lemmas~\ref{lem:u2u3} and~\ref{lem:coeffB} as follows:
	\begin{equation}\label{eq78} 
	\begin{aligned}
		\left \| \sum_{j=1}^{n-2} \mathrm{e}^{-j \tau A}  \mathcal{R}_{n-j}^{[1]}  \right \| 
		& \leq  \tau  \sum_{j=1}^{n-2} \left\| \mathrm{e}^{-j \tau A} A \right\| \left\| b_2(-\tau A)   \right\| \| A^{-1} B A^\beta \|  \times   \\
		&  \int_0^{c_2 \tau}\int_0^\xi  \left \| \mathrm{e}^{-\left( c_2 \tau-\xi \right) A} A^{\gamma - \beta}  \right\| \| A^{-\gamma} B \| \frac{(\xi - \sigma)^2}{2}   \left \| u^{  \prime \prime \prime}\left(t_{n-j-1}+\sigma\right) \right \|  \mathrm{d} \sigma   \mathrm{d} \xi  \\					
		& \leq 	C \tau  \sum_{j=1}^{n-2} t_{j}^{-1} t_{n-j-1}^{\alpha-2} 	\int_0^{c_2 \tau} \int_0^\xi \left( c_2 \tau-\xi \right)^{\beta-\gamma} \frac{(\xi - \sigma)^2}{2}  \mathrm{d} \sigma   \mathrm{d} \xi  \\ & \leq	C t_n^{-1+\zeta+\alpha} \tau^{3-\gamma+\beta -\zeta} ,
	\end{aligned}
\end{equation}
{\allowdisplaybreaks
	\begin{equation}\label{eq24}
		\begin{aligned}
			& \left \| \sum_{j=1}^{n-2} \mathrm{e}^{-j \tau A}  \mathcal{R}_{n-j}^{[5]}   \right \|  \leq  \sum_{j=1}^{n-2} \left\| \mathrm{e}^{-j \tau A} A^{1-\zeta} \right \| \times 
			\\ & \qquad \int_0^{ \tau} \int_0^\xi  \left\| \mathrm{e}^{-\left(  \tau-\xi \right) A} A^{\zeta} \right \|    \frac{ (\xi - \sigma)^3}{6} \left \| A^{-1}BA^\beta \right \|  \left \| (A-B)^{-\beta} u^{\prime \prime \prime \prime}\left(t_{n-j-1}+\sigma\right) \right \|  \mathrm{d} \sigma   \mathrm{d} \xi  \\					 
			&  \qquad \qquad \qquad \quad \, \, \, \, \, \leq	C  \sum_{j=1}^{n-2} t_{j}^{\zeta-1} t_{n-j-1}^{\gamma+\beta-3} \int_0^{ \tau} \int_0^\xi  \frac{ (\xi - \sigma)^3}{6} \left(  \tau-\xi \right)^{-\zeta} \mathrm{d} \sigma   \mathrm{d} \xi \\
			&   \qquad \qquad \qquad \quad \, \, \, \, \, \leq 	C \tau^{3+\gamma+\beta-3\zeta} \sum_{j=1}^{n-2} t_{j}^{\zeta-1} t_{n-j-1}^{-1 +\zeta}    \leq 	C t_n^{-1+2\zeta } \tau^{2+\gamma+\beta-3\zeta}  .
		\end{aligned}
	\end{equation}
}Combining the three cases $j=0$, $j=n-1$, and the interior sum, we obtain
\begin{equation}\label{eq:VII-bound}
	\|\mathrm{VI}\|\le 	C t_n^{-1+2\zeta } \tau^{2+\gamma+\beta-3\zeta}.
\end{equation}
		
\paragraph{Group VII}
We observe that group VII has a structure very similar to that of $T_n^{[1]}$ in V. The difference is that it contains an additional factor $a_{ij}(-\tau A)B$, while the power of $\tau$ is $\tau^4$ instead of $\tau^3$ in $T_n^{[1]}$. The term $a_{ij}(-\tau A)B$ can be estimated using Lemma \ref{lem:coeffB}, resulting in a loss of a factor $\tau^{-\gamma}$. However, we gain an additional factor $\tau$, so that, overall, the estimate for VII gains a factor $\tau^{1-\gamma}$ compared with the estimate for $T_n^{[1]}$ in V. For example, we consider the case $l=0$. Applying Lemmas \ref{lem:coeffB} and \ref{lem:bd4}, we  obtain
\[
\begin{aligned}
	\tau^4 \| W_n\| &\leq 
	\sum_{i=3}^4 \sum_{j=2}^{i-1} \| b_i(-\tau A) \| \,
	\| B \, a_{ij}(-\tau A) \| \,
	\| B \, \mathcal J_{2,j} \, B u'(t_{n-1}) \|  \\
	&\leq C\,t_{n-1}^{-1+\alpha}\,\tau^{\,4+\beta-2\gamma}.
\end{aligned}
\]
In comparison, the estimate for $T_n^{[1]}$ in \eqref{eq:t1t2} is
\[
\begin{aligned}
	\tau^3 \| T_n^{[1]} \| \leq
	Ct_{n-1}^{-1+\alpha}\,\tau^{3+\beta-\gamma}.
\end{aligned}
\]
Thus, the estimate for $\tau^4 \| W_n \|$ is of order $1-\gamma$ higher than that for $\tau^3 \| T_n^{[1]} \|$. The same argument applies to all three cases: $l=0$, $l=n-1$, and $l\neq 0,n-1$.  Therefore, we omit the details.

We are now ready to state the main result.
\begin{theorem}\label{theorem:main}
Let Assumptions~\ref{ass:sectorial}--\ref{ass:beta} hold. Additionally, assume that there exists some $0<\alpha \leq 1/2$ such that
$u_0\in D(A^{1+\alpha})$. Suppose that order conditions~1--4 and~6 of
Table~\ref{tab:order4} hold in strong form, while conditions~5 and~7--9 hold
in the \emph{very weak form} with $A=0$.
Then the numerical solution of \eqref{eq:model} obtained using a four-stage,
fourth-order exponential Runge--Kutta method \eqref{eq:scheme4} satisfies
	\begin{equation}\label{eq:conjbound}
		\|u_n-u(t_n)\|\le C\,\min\{t_n^{-1+\alpha+\zeta},\,t_n^{-1+\zeta},t_n^{-1+2\zeta }\}\,
		\min\big\{\tau^{\,3+\beta-\gamma-\zeta},\ \tau^{\,2+\alpha+2\beta-\zeta},\ \tau^{2+\gamma+\beta-3\zeta} \big\}
	\end{equation}
	uniformly for $0<n\tau\le T$, with $C$ depending on $T$ but not on $n,\tau$.
\end{theorem}

\begin{proof}
	Taking norms in \eqref{eq:e_n-full} and applying the triangle inequality to
	each group, we obtain
	\[
	\begin{aligned}
		\|e_n\| 
		&\leq C\tau\sum_{j=0}^{n-2}t_{n-j-1}^{-\gamma}\|e_j\|
		+ C\tau^{1-\gamma}\|e_{n-1}\|
		+ C\tau^{2-\gamma}\sum_{j=0}^{n-1}\|e_j\| \\
		&\quad
		+ C\tau^{3-2\gamma}\sum_{j=0}^{n-1}\|e_j\|
		+ C\tau^4\sum_{j=0}^{n-1}\|e_j\| \\
		&\quad
		+ C\,\min\{t_n^{-1+\alpha+\zeta},\,t_n^{-1+\zeta},\,t_n^{-1+2\zeta}\}\,
		\min\big\{
		\tau^{\,3+\beta-\gamma-\zeta},\,
		\tau^{\,2+\alpha+2\beta-\zeta},\,
		\tau^{\,2+\gamma+\beta-3\zeta}
		\big\}.
	\end{aligned}
	\]
	We apply the same technique as in Theorem 3.3 of \cite{orig} to bound
	the terms in groups I--IV, which yields
	\[
	\begin{aligned}
		\|e_n\|
		&\leq C\tau\sum_{j=1}^{n-1}t_{n-j}^{-\gamma}\|e_j\|
		+ C\tau^{1-\gamma}\|e_{n-1}\| \\
		&\quad
		+ C\,\min\{t_n^{-1+\alpha+\zeta},\,t_n^{-1+\zeta},\,t_n^{-1+2\zeta}\}\,
		\min\big\{
		\tau^{\,3+\beta-\gamma-\zeta},\,
		\tau^{\,2+\alpha+2\beta-\zeta},\,
		\tau^{\,2+\gamma+\beta-3\zeta}
		\big\}.
	\end{aligned}
	\]
	The proof is completed by applying a discrete Gronwall lemma \cite{HO2010}.
\end{proof}

\section{Numerical investigations}\label{sec:numerics}

 We consider the following one-dimensional advection--diffusion problem
\begin{equation}
	\partial_t u(t,x) - 0.2\,\partial_{xx}u(t,x) = \partial_x u(t,x), \qquad (t,x)\in[0,1]\times[0,1],
\end{equation}
with homogeneous Dirichlet boundary conditions and initial datum $u(0,x) = \sin (\pi x) \in D(A^n)$, $\forall n \in \mathbb{N}_{\geq 1}$, to fulfill part of the requirements of Theorem \ref{theorem:main}. The operator $A$ and $B$ are discretized in space by the standard second-order central finite-difference scheme
on a uniform grid with $n=399$ interior points, resulting in a mesh size of $h = 1/400$. Consequently, the reported error plots reflect only the temporal discretization error of the exponential Runge–Kutta method. The spatial discretization error is embedded in the semi-discrete model and is not separately assessed here. This produces the non-commuting discrete
operators $A = -0.2 L$ (with $L$ the standard second-difference matrix) and $B$ the skew-symmetric central-difference approximation of
$\partial_x$. In particular, the advective term is approximated as follows
\[
\partial_x u \approx \frac{u_{i+1} - u_{i-1}}{2h},
\]
while the diffusive term is discretized as
\[
\partial_{xx} u \approx \frac{u_{i+1} - 2u_i + u_{i-1}}{h^2}.
\]
This results in the discrete matrices $A$ and $B$, respectively. Note that the operator $A$ includes the homogeneous Dirichlet boundary conditions; therefore, the operators $A$ and $B$ do not commute in both the continuous and discrete cases. Additionally, the operators \( A^{-\frac{1}{2}}B \) and \( BA^{-\frac{1}{2}} \) are bounded and adhere to Assumption \ref{ass:B} with \( \gamma = \frac{1}{2} \). 

The implementation of exponential Runge--Kutta methods involves approximating the application of a matrix function to a vector. This is typically done by computing a weighted sum of functions \( \varphi_i(-\tau A) \)applied to vectors, written as \( \sum_{i=1}^k \varphi_i(-\tau A) v_k \).  To improve efficiency, an augmented matrix approach is used (see \cite{doi:10.1137/100788860}). This transforms the sum into a single matrix exponential acting on a vector, represented as \( \operatorname{e}^{-\tau \tilde{A}} V_0 \). The reference solution is computed as $u(T) = \operatorname{e}^{-T(A-B)} u_0$ using MATLAB's built-in \texttt{expm} function. For completeness, we specify the discrete norms used to measure the
numerical errors. For a grid function $v=(v_i)_{i=1}^N$ on a uniform
mesh with mesh width $h$, we define 
\[
\|v\|_{1,h}
= h\sum_{i=1}^N |v_i|,
\qquad
\|v\|_{2,h}
= \left(h\sum_{i=1}^N |v_i|^2\right)^{1/2},
\qquad
\|v\|_{\infty,h}
= \max_{1\leq i\leq N}|v_i|.
\] 

In our numerical experiments, we focus on Krogstad's classical four-stage scheme, given by the butcher tableau
   
\begin{equation}
	\begin{array}{c|cccc}
		0 & & & & \\[6pt]
		\frac{1}{2} & \frac{1}{2}\varphi_{1,2} & & & \\[4pt]
		\frac{1}{2} & \frac{1}{2}\varphi_{1,3} - \varphi_{2,3} & \varphi_{2,3} & & \\[4pt]
		1 & \varphi_{1,4} - 2\varphi_{2,4} & \varphi_{2,4} & \varphi_{2,4} & \\[4pt]
		\hline
		\rule{0pt}{14pt} & \varphi_1 - 3\varphi_2 + 4\varphi_3 & 2\varphi_2 - 4\varphi_3 & 2\varphi_2 - 4\varphi_3 & 4\varphi_3 - \varphi_2
	\end{array}
\end{equation}
as well as the classical fourth-order scheme of Strehmel and Weiner, 
\begin{equation}
	\begin{array}{c|cccc}
		0 & & & & \\[6pt]
		\frac{1}{2} & \frac{1}{2}\varphi_{1,2} & & & \\[4pt]
		\frac{1}{2} & \frac{1}{2}\varphi_{1,3} - \frac{1}{2}\varphi_{2,3} & \frac{1}{2}\varphi_{2,3} & & \\[4pt]
		1 & \varphi_{1,4} - 2\varphi_{2,4} & 0 & 2\varphi_{2,4} & \\[4pt]
		\hline
		\rule{0pt}{14pt} & \varphi_1 - 3\varphi_2 + 4\varphi_3 & 2\varphi_2 - 4\varphi_3 & 2\varphi_2 - 4\varphi_3 & 4\varphi_3 - \varphi_2
	\end{array}.
\end{equation}
Here, $\varphi_{i,j}$ denotes $\varphi_i(-c_j z)$, where $c_2 = c_3 = \frac{1}{2}$ and $c_4 = 1$, and $\varphi_i = \varphi_i(-z)$).~Both methods satisfy the fourth-order conditions in accordance with the order conditions stated in Theorem~\ref{theorem:main}. The present
experiment isolates the fourth-order scheme against the reference slope $11/4 = 2.75$, the
order predicted by Theorem \ref{theorem:main}.
 
\begin{figure}[H]
	\centering 
\includegraphics[width=0.65\linewidth]{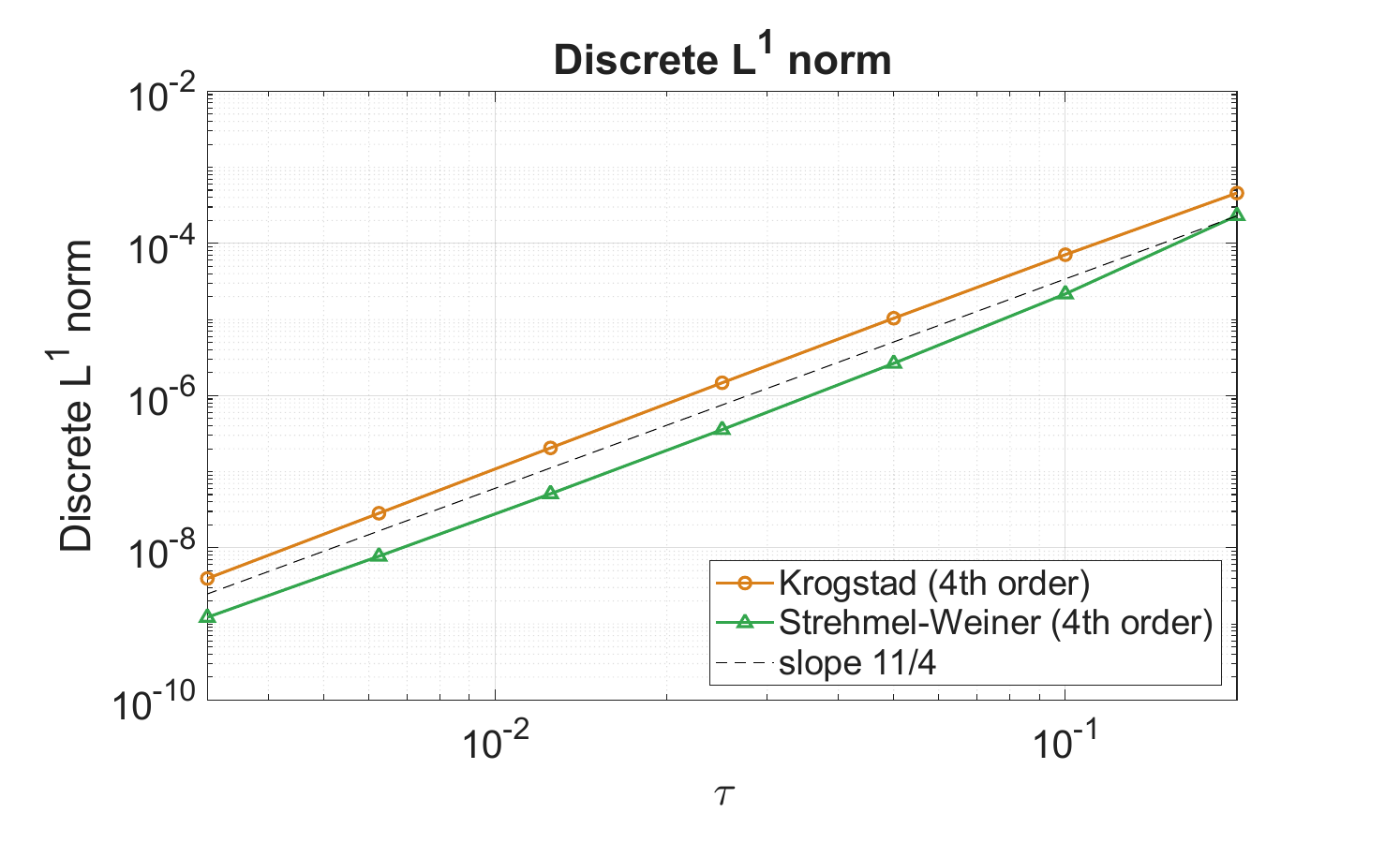} \\[1ex]
\includegraphics[width=0.65\linewidth]{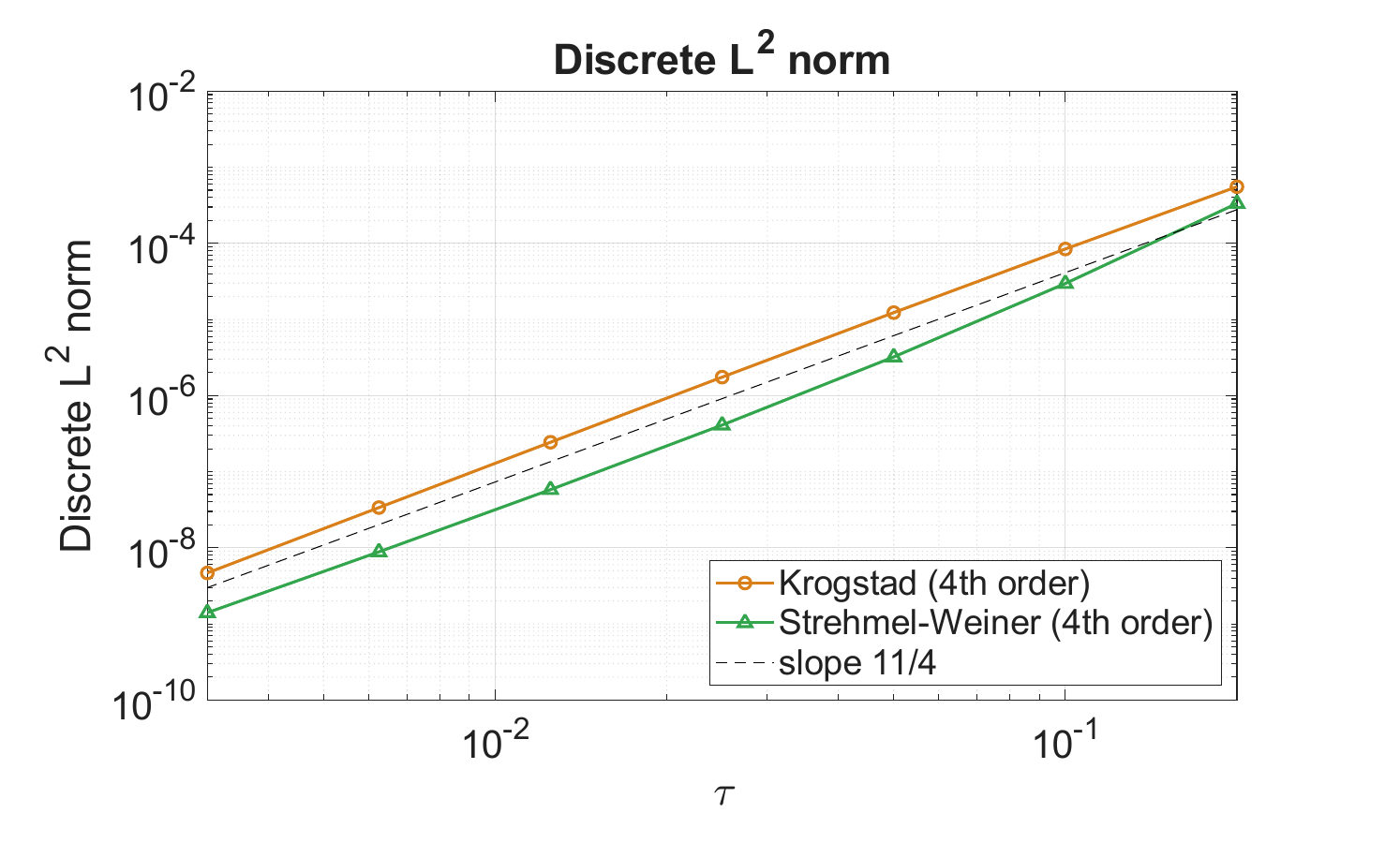} \\[1ex]
\includegraphics[width=0.65\linewidth]{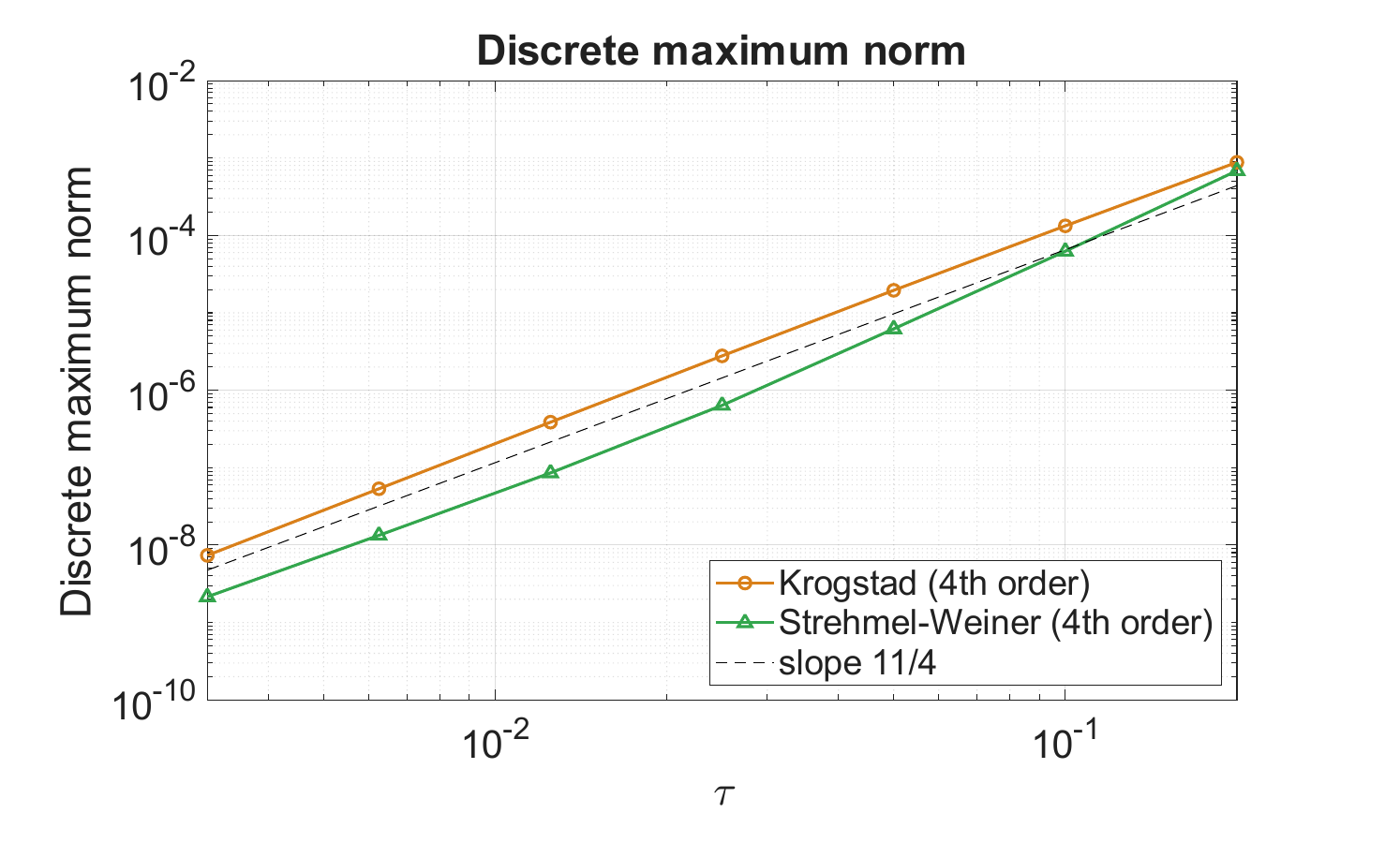}
		\caption{Global error of the fourth-order Krogstad and
			Strehmel--Weiner schemes versus the time step $\tau$, measured in the
			discrete $\|\cdot\|_{1,h}$ (top), $\|\cdot\|_{2,h}$ (middle), and
			$\|\cdot\|_{\infty,h}$ (bottom) norms with respect to the reference
			solution $\mathrm{e}^{-(A-B)}u_0$. The dashed line in each panel
			indicates the reference slope $11/4$.}
	\label{fig:orders}
\end{figure}

As shown in Figure~\ref{fig:orders}, the numerical error curves concentrate almost exactly around the fractional slope of 11/4 = 2.75 across all three discrete norms. Notably, since $\zeta$ in Theorem \ref{theorem:main} can be chosen arbitrarily small, the order reduction induced by the factor $-\zeta$ can be made arbitrarily small as well. Consequently, the resulting loss of accuracy is negligible in practice. This precisely matches the theoretical prediction formulated in Theorem~\ref{theorem:main} for this problem, thereby confirming the sharpness of our error bounds and the practical occurrence of order reduction. 

To examine whether the temporal convergence behavior of the Krogstad
scheme is affected by spatial refinement, we repeat the numerical
experiment on the meshes
\[
h=\frac{1}{200},\qquad
\frac{1}{400},\qquad
\frac{1}{800},\qquad
\frac{1}{1600}.
\]
For each fixed \(h\), the time step is successively refined while the
temporal error is measured with respect to the reference solution of the
corresponding semidiscrete problem. More precisely, at \(t_n=n\tau\), we
set
\[
E_h(\tau)
:=
\left\|u_n-\mathrm{e}^{-t_n(A-B)}u_0\right\|,
\]
where \(A\), \(B\), and \(u_0\) denote the corresponding spatially
discretized quantities for the mesh size \(h\).

\begin{figure}[h!]
	\subfigure{\includegraphics[width=0.53\textwidth]{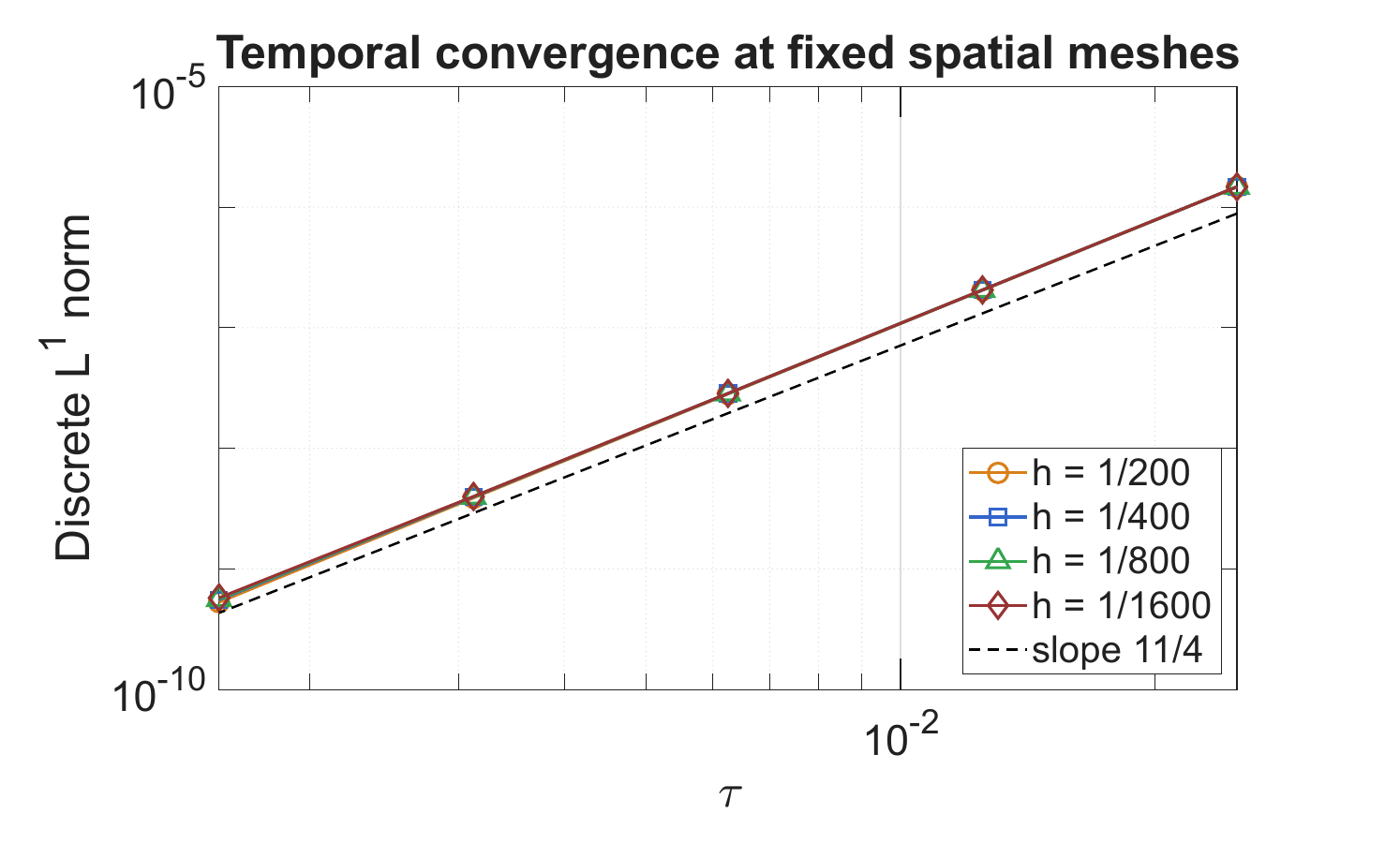}}
	\subfigure{\includegraphics[width=0.53\textwidth]{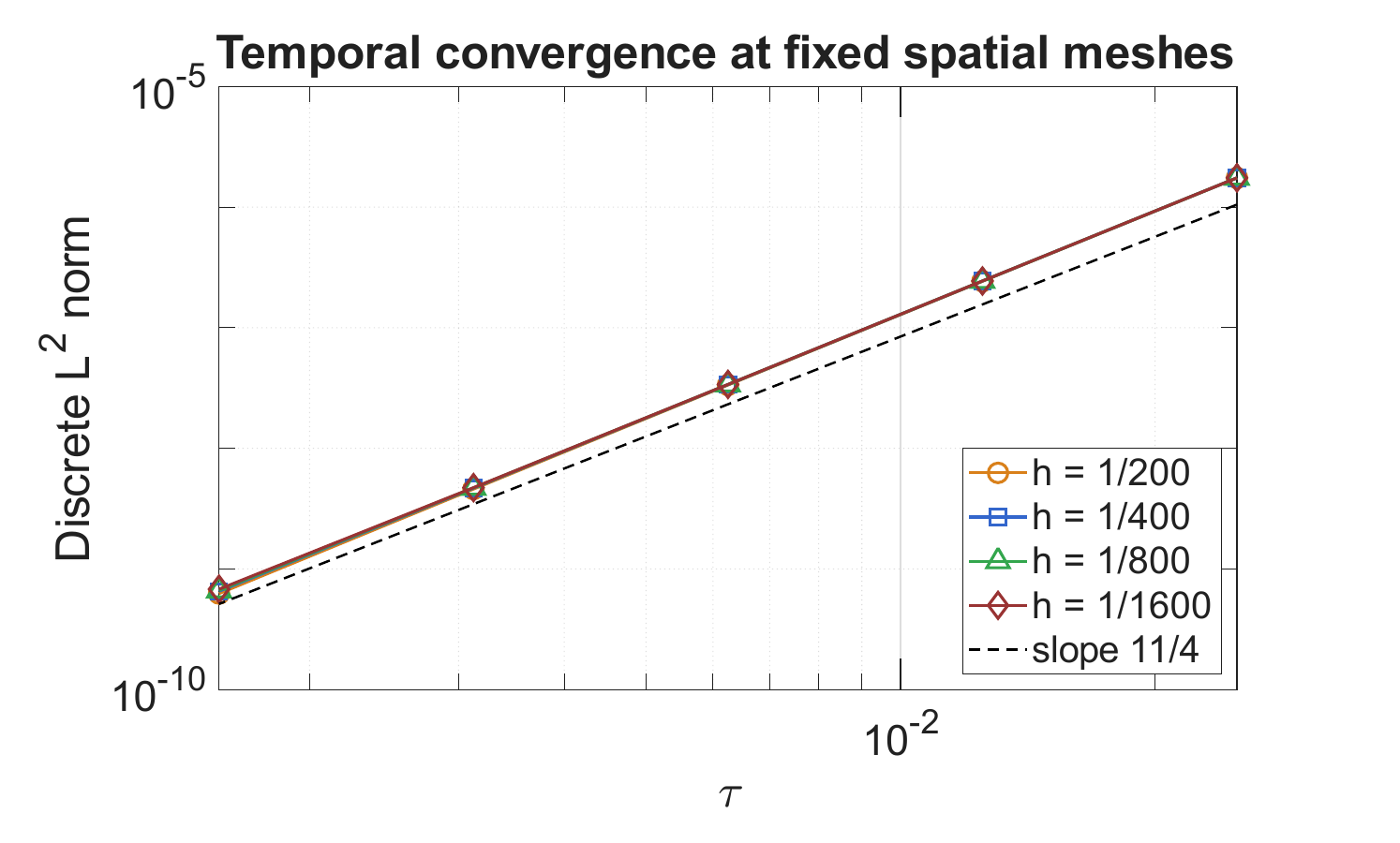}}
	\begin{center}
		\subfigure{\includegraphics[width=0.53\textwidth]{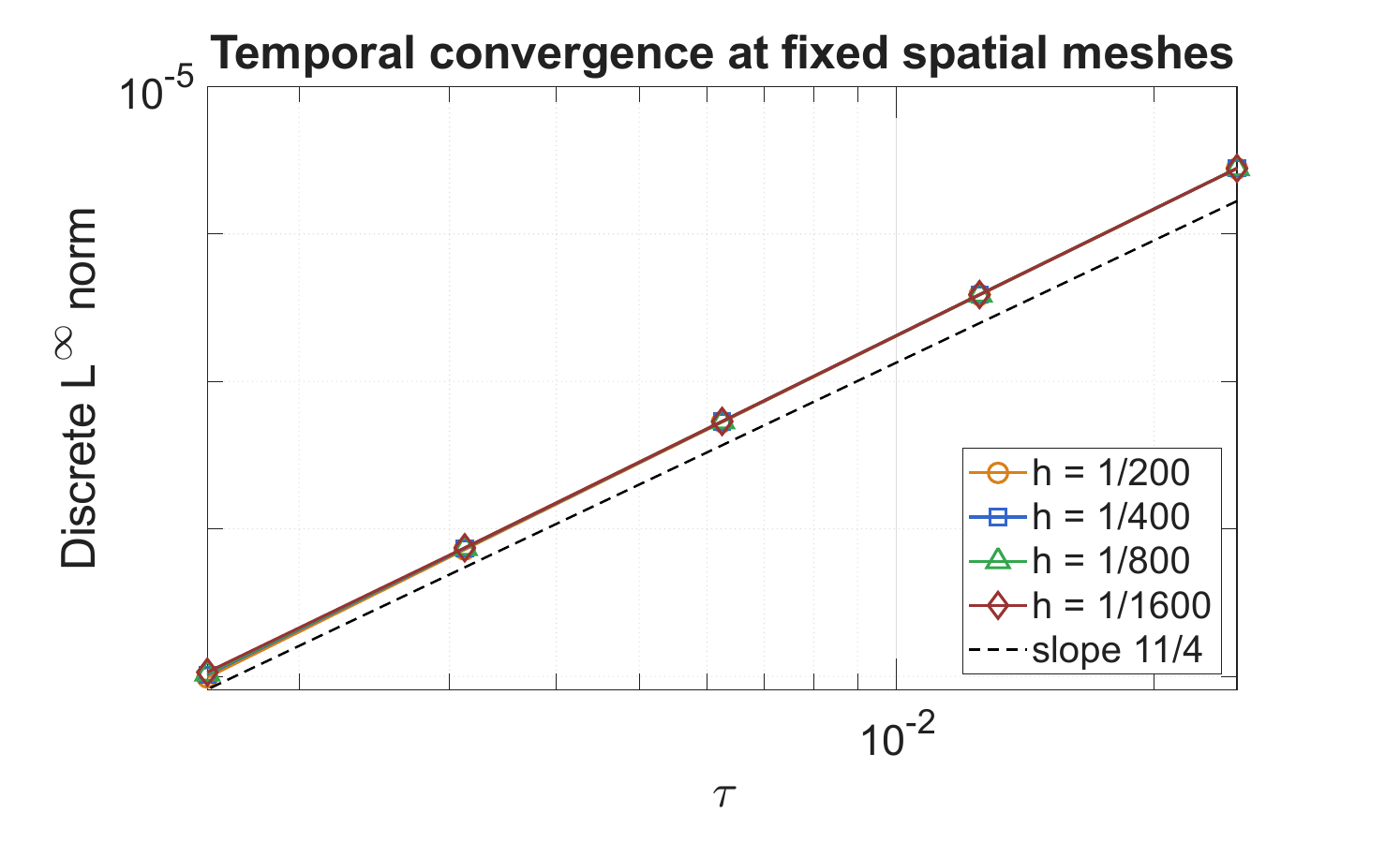}}
	\end{center}
\caption{Temporal convergence of the Krogstad scheme in the discrete
	$L^1(\Omega)$, $L^2(\Omega)$, and $L^\infty(\Omega)$ norms for different
	spatial mesh sizes $h$. The dashed lines indicate the reference slope $11/4$.}   
	\label{fig:mesh_refinement}
\end{figure}

Figure \ref{fig:mesh_refinement} displays the resulting
temporal errors in the discrete \(L^1(\Omega)\), \(L^2(\Omega)\), and
\(L^\infty(\Omega)\) norms, respectively. The convergence curves exhibit
the same asymptotic behavior as the spatial mesh is refined and are
consistent with the reference rate \(11/4\).

We additionally quantify the dependence of the error prefactor on the
spatial mesh. For each norm and each \(h\), let
\[
C_h^{\mathrm{num}}
=
\max_{\tau\in\mathcal{T}}
\tau^{-11/4} E_h(\tau),
\] 
where \(\mathcal{T}\) is the set of time-step sizes used in the
experiment. The computed numerical error constants for both the Krogstad and Strehmel--Weiner schemes are summarized in Table~\ref{tab:mesh_constants}.

\begin{table}[H]
	\centering
	\caption{Numerical error constants \(C_h^{\mathrm{num}}\) for the
		Krogstad and Strehmel--Weiner schemes under spatial mesh refinement.}
	\label{tab:mesh_constants}
	\begin{tabular}{cc|cccc}
		\hline
		Scheme & Norm
		& \(h=1/200\) & \(h=1/400\) & \(h=1/800\) & \(h=1/1600\)\\
		\hline
		Krogstad
		& \(L^1(\Omega)\)
		& \(3.7371\times10^{-2}\)
		& \(3.7431\times10^{-2}\)
		& \(3.7447\times10^{-2}\)
		& \(3.7451\times10^{-2}\)\\
		
		& \(L^2(\Omega)\)
		& \(4.4336\times10^{-2}\)
		& \(4.4406\times10^{-2}\)
		& \(4.4424\times10^{-2}\)
		& \(4.4429\times10^{-2}\)\\
		
		& \(L^\infty(\Omega)\)
		& \(7.0682\times10^{-2}\)
		& \(7.0790\times10^{-2}\)
		& \(7.0818\times10^{-2}\)
		& \(7.0825\times10^{-2}\)\\
		\hline
		Strehmel--Weiner
		& \(L^1(\Omega)\)
		& \(9.0781\times10^{-3}\)
		& \(1.0315\times10^{-2}\)
		& \(1.0792\times10^{-2}\)
		& \(1.1485\times10^{-2}\)\\
		
		& \(L^2(\Omega)\)
		& \(1.0493\times10^{-2}\)
		& \(1.1936\times10^{-2}\)
		& \(1.2494\times10^{-2}\)
		& \(1.3303\times10^{-2}\)\\
		
		& \(L^\infty(\Omega)\)
		& \(1.6204\times10^{-2}\)
		& \(1.8406\times10^{-2}\)
		& \(1.9270\times10^{-2}\)
		& \(2.0525\times10^{-2}\)\\
		\hline
	\end{tabular}
\end{table}

For the
Krogstad scheme, the values of \(C_h^{\mathrm{num}}\) are essentially
unchanged under spatial refinement and approach approximately
\(3.75\times10^{-2}\), \(4.44\times10^{-2}\), and
\(7.08\times10^{-2}\) in the discrete \(L^1(\Omega)\),
\(L^2(\Omega)\), and \(L^\infty(\Omega)\) norms, respectively.

For the Strehmel--Weiner scheme, a moderate increase in
\(C_h^{\mathrm{num}}\) is observed as the mesh is refined. Nevertheless,
the values remain of the same magnitude and bounded over the range of
spatial meshes considered. In particular, for the finest mesh
\(h=1/1600\), the values are \(1.1485\times10^{-2}\),
\(1.3303\times10^{-2}\), and \(2.0525\times10^{-2}\) in the three norms,
respectively. Thus, the numerical experiments do not indicate a
deterioration of the temporal error constants under the spatial refinements
considered.

The mesh-refinement convergence plots for the Strehmel--Weiner scheme show
qualitatively similar behavior and are therefore omitted for brevity.

\section{Conclusion}\label{sec:conclusion}

In this paper, we extended the order-reduction analysis of \cite{orig} from
third- to fourth-order explicit exponential Runge--Kutta methods for
$u'+Au=Bu$ with non-commuting, unbounded $A$ and $B$. We explicitly derived and solved the defect recursions for a general four-stage scheme, identifying all error terms specific to fourth-order methods. We then proved tight bounds for these terms and confirmed our findings through numerical investigation of an advection–diffusion problem. For this class of problems, a fourth-order exponential
Runge--Kutta method\cite{orig} once $A$ and $B$ fail to commute: both are capped, in our analysis and in the
experiment, at order $11/4$. 

Two main questions remain open: first, whether $11/4$ is a strict upper bound for \emph{every} explicit exponential Runge--Kutta method of order $\ge 3$ applied to non-commuting $A$ and $B$ of this type, or merely a property of the specific order conditions studied here and in \cite{orig}. Because fourth-order methods can satisfy the order conditions in different ways, their resulting convergence orders may differ from those presented here. Another open issue is the necessity of adding an extra stage to avoid order reduction, even under stronger regularity assumptions on $B$. Resolving these questions and designing a new method that preserves full fourth-order convergence will be the focus of future work.
	
\subsection*{Funding} 
Thi Tam Dang has been supported by the Jane and Aatos Erkko Foundation. Pablo Alexei Gazca-Orozco and Trung Hau Hoang gratefully acknowledge support from the Charles University Research program no.\ PRIMUS/25/SCI/025 and UNCE/24/SCI/005. 

\subsection*{Conflict of interest}
The authors declare that there are no conflicts of interest.

\subsection*{Data availability}
No data was used for the research described in the article.


\begin{thebibliography}{99}
	
	\bibitem{orig} 
	T.-H. Hoang. 
	Order reduction of exponential Runge--Kutta methods: non-commuting operators. To appear 
	\emph{Adv. Appl. Math. Mech.}, 2026. 
	
	\bibitem{CoxMatthews2002}
	S.~M. Cox and P.~C. Matthews.
	Exponential time differencing for stiff systems.
	\emph{J. Comput. Phys.}, 176(2):430--455, 2002.
	
	\bibitem{EinkemmerOstermann2014}
	L.~Einkemmer and A.~Ostermann.
	Overcoming order reduction of exponential Runge--Kutta methods for advection--diffusion--reaction equations.
	\emph{J. Comput. Appl. Math.}, 272:78--92, 2014.
	
	\bibitem{GonzalezOstermannThalhammer2006}
	C.~Gonzalez, A.~Ostermann, and M.~Thalhammer.
	A note on error estimates for exponential Runge--Kutta methods for non-autonomous linear problems.
	\emph{BIT Numer. Math.}, 46(3):513--521, 2006.
	
	\bibitem{HO2005} 
	M. Hochbruck and A. Ostermann. 
	Explicit exponential Runge--Kutta methods for semilinear parabolic problems. 
	\emph{SIAM J. Numer. Anal.}, 43(3):1069--1090, 2005.

\bibitem{pazy1983semigroups}
A.~Pazy,
\emph{Semigroups of Linear Operators and Applications to Partial Differential Equations},
Applied Mathematical Sciences, 
Springer, New York, 1983.

	\bibitem{HO2010} 
	M. Hochbruck and A. Ostermann. 
	Exponential integrators. 
	\emph{Acta Numerica}, 19:209--286, 2010.
	
	\bibitem{HochbruckOstermannSchweitzer2009}
	M.~Hochbruck, A.~Ostermann, and J.~Schweitzer.
	Exponential Rosenbrock-type methods.
	\emph{SIAM J. Numer. Anal.}, 47(1):785--803, 2009.

\bibitem{henry1981geometric}
D.~Henry, \emph{Geometric Theory of Semilinear Parabolic Equations}, Lecture Notes in Mathematics, Springer-Verlag, 1981.

\bibitem{HOCHBRUCK2005323}
M.~Hochbruck and A.~Ostermann. {Exponential {R}unge–{K}utta methods for parabolic problems}. Appl. Numer. Math., 53(2):323--339, 2005.

	\bibitem{Ablowitz1979}
M.~J. Ablowitz and A.~Zeppetella.
\newblock Explicit solutions of Fisher's equation for a special wave speed.
\newblock {\em Bulletin of Mathematical Biology}, 41(6):835--840, 1979.

\bibitem{grindrod1996theory}
P.~Grindrod.
\newblock {\em The Theory and Applications of Reaction-Diffusion Equations: Patterns and Waves}.
\newblock Clarendon Press, Oxford, 1996.

\bibitem{bailyn1994survey}
M.~Bailyn.
\newblock {\em A Survey of Thermodynamics}.
\newblock American Institute of Physics, 1994.

\bibitem{bergman2011fundamentals}
T.~L. Bergman.
\newblock {\em Fundamentals of Heat and Mass Transfer}.
\newblock Wiley, 2011.


\bibitem{Hau1}
T.~T.~Dang and T.-H.~Hoang, {How to avoid order reduction in third-order exponential {R}unge--{K}utta methods for problems with non-commutative operators?} ArXiv preprint (arXiv:2412.11920), 2024.

\bibitem{Hau2}
L.~Einkemmer, T.-H.~Hoang, and A.~Ostermann, {Should exponential integrators be used for advection-dominated problems?} \emph{Adv. Appl. Math. Mech.}, 2026.

\bibitem{ALLEN19791085}
S.~M. Allen and J.~W. Cahn.
\newblock A microscopic theory for antiphase boundary motion and its application to antiphase domain coarsening.
\newblock {\em Acta Metallurgica}, 27(6):1085--1095, 1979.

\bibitem{Bronsard1993}
L.~Bronsard and F.~Reitich.
\newblock On three-phase boundary motion and the singular limit of a vector-valued {G}inzburg–{L}andau equation.
\newblock {\em Archive for Rational Mechanics and Analysis}, 124(4):355--379, 1993.

	\bibitem{Krogstad} 
	S. Krogstad. 
	Generalized integrating factor methods for stiff PDEs.
	\emph{J. Comput. Phys.}, 203(1):72--88, 2005.
	
	\bibitem{Lawson1967}
	J.~D. Lawson.
	Generalized Runge-Kutta processes for solving $y'= f(x, y)$.
	\emph{SIAM J. Numer. Anal.}, 4(3):372--380, 1967.
	
	\bibitem{Hau3}
	T.~T.~Dang and T.-H.~Hoang, {From memory model to CPU time: exponential integrators for advection-dominated problems}. ArXiv preprint arXiv:2512.03679, 2025. 
	
	\bibitem{LuanOstermann2013}
	V.~T. Luan and A.~Ostermann.
	Explicit exponential Runge--Kutta methods of high order for parabolic problems.
	\emph{J. Comput. Appl. Math.}, 256:168--179, 2013.
	
	\bibitem{LuanOstermann2014}
	V.~T. Luan and A.~Ostermann.
	Stiff order conditions for exponential Runge--Kutta methods of order five.
	\emph{J. Comput. Appl. Math.}, 268:139--152, 2014.
	
	\bibitem{MinchevWright2005}
	B.~Minchev and W.~M. Wright.
	\emph{A review of exponential integrators for first order semi-linear problems}.
	Technical Report 2/05, Department of Mathematics, NTNU, Norway, 2005.
	
	\bibitem{OstermannRoche1992}
	A.~Ostermann and M.~Roche.
	Order reduction of Runge--Kutta methods for time-dependent boundary conditions.
	\emph{BIT Numer. Math.}, 32(4):685--699, 1992.
	
	\bibitem{Pazy} 
	A. Pazy. 
	\emph{Semigroups of Linear Operators and Applications to Partial Differential Equations}. 
	Springer-Verlag, New York, 1983.

\bibitem{CaliariOstermann2009}
M.~Caliari and A.~Ostermann.
Implementation of exponential Rosenbrock-type methods.
\emph{Appl. Numer. Math.}, 59(3-4):568--581, 2009.

\bibitem{HochbruckOstermann2005_parabolic}
M.~Hochbruck and A.~Ostermann.
Exponential Runge--Kutta methods for parabolic problems.
\emph{Appl. Numer. Math.}, 53(2-4):323--339, 2005.

\bibitem{doi:10.1137/100788860}
A. H. Al-Mohy and N. J. Higham.
Computing the action of the matrix exponential, with an application to exponential integrators.
\textit{SIAM Journal on Scientific Computing} 33, 488--511, 2011.

\bibitem{RainwaterTokman2014}
G.~Rainwater and M.~Tokman.
A new class of exponential integrators for stiff systems.
\emph{J. Comput. Phys.}, 269:285--303, 2014.

\end{thebibliography}
\end{document}